\documentclass[english,twoside]{article}
\usepackage{amsmath,amssymb}
\usepackage{latexsym}
\usepackage[english]{babel}
\usepackage{amsthm}
\usepackage[dvips]{graphicx}
\usepackage[pdftex]{color}

\newcommand{\N}{\mathbb{N}}
\newcommand{\Z}{\mathbb{Z}}

\newcommand{\R}{\mathbb{R}}

\newcommand{\A}{\mathbb{A}}

\newcommand{\Ac}{\overline{\mathbb{A}}}
\newcommand{\s}{\mathbb{S}}
\newcommand{\D}{\mathbb{D}}

\newcommand{\tore}{\mathbb{T}}
\newcommand{\T}[1]{\tore^{#1}}

\newcommand{\abs}[1]{\left|#1\right|}

\newcommand{\fonc}[3]{#1:#2\to#3}

\DeclareMathOperator{\rot}{Rot}

\DeclareMathOperator{\diam}{diam}
\DeclareMathOperator{\dom}{dom}

\DeclareMathOperator{\sing}{Sing}
\DeclareMathOperator{\fix}{Fix}
\DeclareMathOperator{\supp}{Supp}

\theoremstyle{theorem}
\newtheorem{lemma}{Lemma}[section]
\newtheorem{coro}[lemma]{Corollary}
\newtheorem{prop}[lemma]{Proposition}
\newtheorem{theo}[lemma]{Theorem}

\newtheorem*{theoa}{Theorem A}

\newtheorem*{theob}{Theorem B}

\newtheorem*{theoc}{Theorem C}

\newtheorem*{theod}{Theorem D}

\theoremstyle{remark}

\newtheorem{rema}{Remark}

\newtheorem*{claim}{Claim}

\title{Applications of Forcing Theory to  Homeomorphisms of the Closed Annulus}
\author{Conejeros, Jonathan and Tal, F\'abio Armando}
\date{\today}
\begin{document}

\maketitle

\begin{abstract}
 This paper studies homeomorphisms of the closed annulus that are isotopic to the identity from the viewpoint of rotation theory, using a newly developed forcing theory for surface homeomorphisms. Our first result is a solution to the so called strong form of Boyland's Conjecture on the closed annulus: Assume $f$ is a homeomorphism of $\Ac:=(\R/\Z)\times [0,1]$ which is isotopic to the identity and preserves a Borel probability measure $\mu$ with full support. We prove that if the rotation set of $f$ is a non-trivial segment, then the rotation number of the measure $\mu$ cannot be an endpoint of this segment. We also study the case of homeomorphisms such that $\A=(\R/\Z)\times (0,1)$ is a region of instability of $f$. We show that, if the rotation numbers of the restriction of $f$ to the boundary components lies in the interior of the rotation set of $f$, then $f$ has uniformly bounded deviations from its rotation set. Finally, by combining this last result and recent work on realization of rotation vectors for annular continua, we obtain that if $f$ is any area-preserving homeomorphism of $\Ac$ isotopic to the identity, then  for every real number $\rho$ in the rotation set of $f$, there exists an associated Aubry-Mather set, that is, a compact $f$-invariant set such that every point in this set has a rotation number equal to $\rho$. This extends a result by P. Le Calvez previously known only for diffeomorphisms.
\end{abstract}

\section{Introduction}

This article studies homeomorphisms of the closed annulus that preserve the orientation and the boundary components, by the point of view of rotation theory. We denote by $\T{1}=\R/\Z$ the circle, by $\Ac=\T{1}\times [0,1]$ the closed annulus and by $\widehat{\Ac}=\R\times [0,1]$ its universal covering. Let $\widehat{\pi}:\widehat{\Ac}\to\Ac$ the corresponding covering map, and $p_1:\widehat{\Ac}\to\R$ the projection on the first coordinate. Let $f:\Ac\to\Ac$ be a homeomorphism which preserves both orientation and boundary components and let $\widehat{f}$ be a lift of $f$ to the universal covering. Inspired by the concept of Poincar\'e's rotation number for orientation-preserving homeomorphisms of the circle, one can define a similar object for $\widehat{f}$, called the {\em rotation set of $\widehat{f}$}, as follows: let $\mu$ be an $f$-invariant Borel probability measure on $\Ac$. We can define the {\em rotation number of $\mu$ for $\widehat{f}$} as
$$  \rot(\widehat{f},\mu):= \int_{\Ac} p_1(\widehat{f}(\widehat{z}))-p_1(\widehat{z})\,d\mu(z),  $$
where $\widehat{z}\in \widehat{\pi}^{-1}(z)$. Note that this definition does not depend on the choice of $\widehat{z}\in \widehat{\pi}^{-1}(z)$. The  {\em rotation set of  $\widehat{f}$ }, denoted by $\rot(\widehat{f})$, is the set of all rotation numbers of $f$-invariant Borel probability measures. Since the set of $f$-invariant Borel probability measures is convex and compact in the weak-$\ast$topology, one shows that the rotation set of $\widehat{f}$ is a non-empty compact interval of $\R$.

We remark that the concept of rotation sets is not restricted to homeomorphisms of the annulus, and has been useful  in the general study of homeomorphisms in the isotopy class of the identity of surfaces in general, and particularly for the two dimensional torus. One of the reasons for the growing interesse in the subject is the variety of dynamical properties and phenomena that can be deduced from rotation sets; it is a useful tool in, for instance, estimating the topological entropy of a map or determining the existence of periodic points with arbitrarily large prime periods and distinct rotational behavior.

One of the driving problems in the understanding of the rotation theory for homeomorphisms of the closed annulus and of the two dimensional torus $\T{2}=\R^2/\Z^2$  has been the Boyland's Conjecture, see for instance \cite{transitiveannulusI,transitivetorus}. In the original form, Boyland's Conjecture for the closed annulus claimed that, whenever $f:\Ac\to\Ac$ preserved the Lebesgue measure and had a lift $\widehat f$ such that the rotation number of the Lebesgue measure for $\widehat f$ was null, then either the rotation set of $\widehat{f}$ was a singleton, or $0$ lied in the interior of the rotation set of $\widehat{f}$. A stronger version of this conjecture has also been proposed, saying that whenever $f$ preserved the Lebesgue measure and the rotation set of $\widehat f$ was a nondegenerate interval, then the rotation number of the Lebesgue measure for $\widehat f$ always lies in the interior of the rotation set, and similar questions were posed for homeomorphisms of $\T{2}$. In \cite{salvadorplms} the strong form of Boyland's Conjecture for $\T{2}$ was shown to hold for $\mathcal{C}^{1+\epsilon}$-diffeomorphisms, a result later extended for the $\mathcal{C}^0$ case in \cite{lct}. The later paper also proved the original conjecture for the closed annulus, but the strong version remained untenable. Our first result of this paper is the solution to this problem.

\begin{theoa}
  Let $f$ be a homeomorphism of the closed annulus $\Ac:=\T{1}\times [0,1]$ which is isotopic to the identity and preserves a Borel probability measure $\mu$ with full support. Let $\widehat{f}$ be a lift of $f$ to $\R\times [0,1]$. Suppose that $\rot(\widehat{f})$ is a non-trivial segment. Then the rotation number of $\mu$ cannot be an endpoint of $\rot(\widehat{f})$.
\end{theoa}

Another research topic in rotation theory that has gathered substantial attention lately is the concept of bounded rotational deviations from rotation sets. It is a well known fact that, given an orientation-preserving homeomorphism $h:\T{1}\to\T{1}$ and a lift $\widehat{h}$ to the real line whose rotation number is $\alpha$, one has that every orbit of $\widehat{h}$ remains at a bounded distance from the orbit of the associated rigid rotation. That is, there exists some constant $L>0$ such that, for all $\widehat{x}\in\R$ and all $n\in\N$, $\vert \widehat{h}^{n}(\widehat{x})-\widehat{x}-n\alpha\vert\le L$ (and in this case $L$ can be taken as $1$). A natural question is then to ask if some aspects of this property extend to similar situations for homeomorphisms of surfaces. For instance, one could pose the problem: Consider a homeomorphism $f$ of $\T{2}$ in the isotopy class of the identity and say that {\it $f$ has uniformly bounded deviations from its rotation set} if, given $\widehat{f}$ a lift of $f$ to $\R^2$, the universal covering of $\T{2}$, there a constant $L>0$ such that, for all $\widehat{z}\in\R^2$ and all $n\in\N$, if $d$ is the distance between a point and a set of $\R^2$, then $d(\widehat{f}^{n}(\widehat{z})-\widehat{z}, n \rot(\widehat{f}))\le L$. On then asks if it always holds that $f$ has uniformly bounded deviations. This question is false in general, particularly when the rotation set of $\widehat{f}$ is a singleton (see for instance \cite{koropeckikocksard,koropeckitalpams}), but it does hold in many situations, particularly when $\rot(\widehat{f})$ has nonempty interior (see \cite{davalos1,davalos2,salvadorplms,gukotal,boundedunbounded,lct}), and similar results also are valid for homeomorphisms of $\T{2}$ isotopic to Dehn Twists (see \cite{addasgarciatal}). Furthermore, bounded deviations have also shown to have relevant dynamical consequences, for instance it was used in the proof of Boyland's Conjecture on $\T{2}$ in \cite{salvadorplms,lct}. In some particular cases it can also imply that the dynamics factors over ergodic rotations of $\T{2}$ (see \cite{jagerinventiones}) or $\T{1}$ (see \cite{jagertal}).

Our second theorem deals with bounded deviations from rotation sets for homeomorphisms of $\Ac$ in the following relevant scenario.  We will say that $\A=\T{1}\times (0,1)$ is a {\it Birkhoff region of instability} for a homeomorphism $f$ of $\Ac$ if for any neighborhood $U$ of $\T{1} \times \{0\}$ and any neighborhood $V$ of $\T{1} \times \{1\}$ one can find points $x\in U,\, y\in V$ and positive integers $n_1, n_2$ such that $f^{n_1}(x)\in V$ and $f^{n_2}(y)\in U$.
\begin{theob}
    Let $f$ be a homeomorphism of the closed annulus $\Ac=\T{1}\times [0,1]$ which is isotopic to the identity. Suppose that $\A=\T{1}\times (0,1)$ is a Birkhoff region of instability for $f$.  Let $\widehat{f}$ be a lift of $f$ to $\R\times [0,1]$. Suppose that $\rot(\widehat{f})=[\alpha,\beta]$ and that both boundary component rotation numbers are strictly larger than $\alpha$. Then there exists a real constant $L>0$ such that for every $\widehat{z}\in \R\times [0,1]$ and every integer $n\geq 1$ we have
  $$ p_1(\widehat{f}^n(\widehat{z}))-p_1(\widehat{z})-n\alpha \geq -L.$$
  Likewise, if we assume that both boundary component rotation numbers are strictly smaller than $\beta$, then there exists a real constant $L>0$ such that for every $\widehat{z}\in \R\times [0,1]$ and every integer $n\geq 1$ we have
  $$ p_1(\widehat{f}^n(\widehat{z}))-p_1(\widehat{z})-n\beta \leq L. $$
\end{theob}
Interestingly, this is to our knowledge the first positive result on bounded deviations for homeomorphisms of $\Ac$. Note that both the hypotheses that the rotation numbers of the boundary components of $\Ac$ lie in the interior of the rotation set and that $\A$ is a region of instability of Birkhoff cannot be removed. One can easily create examples of homeomorphisms of $\Ac$ that do not present bounded deviation when $\A$ is not a Birkhoff region of instability, and we present in Section \ref{sec:example}, the following example.
\begin{prop}
  There exists a homeomorphism $f$ of the closed annulus $\Ac$ which is isotopic to the identity, such that $\A$ is a Birkhoff region of instability for $f$ and such that $f$ has a lift $\widehat{f}$ to $\R\times [0,1]$ satisfying:
  \begin{itemize}
    \item[(i)] $\rot(\widehat{f})=[0,1]$, and
    \item[(ii)] for every real number $L>0$ there exists a point $\widehat{z}$ in $\R\times [0,1]$ and an integer $n$ such that
  $$ p_1(\widehat{f}^{n}(\widehat{z}))-p_1(\widehat{z}) < -L$$
  \end{itemize}
\end{prop}

The third topic we deal in this paper is of the strong realization of rotation numbers. We say that a point $z\in\Ac$ \textit{has rotation number equal to $\rho$} if, for any $\widehat{z}\in \widehat \pi^{-1}(z)$, one has $\lim_{n\to\infty}p_1(\widehat{f}^{n}(\widehat{z})-\widehat{z})/n=\rho$, and we note that if the limit exists, it is independent of which $\widehat{z}$ one chooses  in $\widehat \pi^{-1}(z)$. We say that a number $\rho\in\rot(\widehat{f})$ \textit{is realized by an ergodic measure} if there exists some $f$-invariant ergodic measure $\nu$ such that $\rot(\widehat{f},\nu)=\rho$. Finally, one says that $\rho$ \textit{is realized by a compact invariant set} if there exists a compact invariant set $Q$ such that all point in $Q$ have rotation number equal to $\rho$. There is a natural hierarchy of realization. Any $\rho$ that is realized by a compact set is also realized by an ergodic measure, any $\rho$ that is realized by an ergodic measure is also the rotation number of some point, and the rotation number of points are clearly contained in $\rot(\widehat{f})$. Note that, if $f$ is an area-preserving twist map of the open annulus, then a ground-breaking result by Mather (see \cite{mather}) shows that every point $\rho$ in the rotation set of $\widehat{f}$ is realized by a compact set, the so called Aubry-Mather set of $\rho$. A natural question is then to decide which points in the rotation set of $\widehat{f}$ were realized by compact subsets.

This turned out to be a difficult problem to tackle. An important result by Handel (see \cite{Handel}) showed that the set of points that are realized by ergodic measures is a closed subset of $\rot(\widehat{f})$ and he further showed that, except for a possible discrete subset, all were also realized by compact invariant sets. Franks (see \cite{franksannals1}) showed that, if $f$ preserves a measure of full support, then every rational number in $\rot(\widehat{f})$ is realized by a periodic orbit and Le Calvez (\cite{LeCalvezInventiones}) showed that, if $f$ is an area-preserving diffeomorphism, then every point in the rotation set is realized by a compact invariant subset. The general question on whether a point in the rotation set of $f$  is always realized by a compact invariant set remains open.

Our third theorem, that relies on Theorem B, shows that the answer to this problem is true for regions of instability of Mather. We will say that $\A=\T{1}\times (0,1)$ is a {\it Mather region of instability} for a homeomorphism $f$ of $\Ac$ if there exists points $z_1, z_2$ in $\A$ such that the $\alpha$-limit set of $z_1$ is contained in $\T{1} \times \{0\}$ and while the $\omega$-limit set of $z_1$ is contained in $\T{1} \times \{1\}$ and such that the $\alpha$-limit set of $z_2$ is contained in $\T{1} \times \{1\}$ and while the $\omega$-limit set of $z_2$ is contained in $\T{1} \times \{0\}$.
\begin{theoc}
 Let $f$ be a homeomorphism of the closed annulus $\Ac=\T{1}\times [0,1]$ which is isotopic to the identity. Suppose that $\A=\T{1}\times (0,1)$ is a Mather region of instability for $f$.  Let $\widehat{f}$ be a lift of $f$ to $\R\times [0,1]$. For every $\rho$ in $\rot(\widehat{f})$ there exists a compact invariant set $Q_\rho$ such that for every point of $Q_\rho$ has a rotation number well-defined and it is equal to $\rho$. Moreover, if $\rho=p/q$ is a rational number, written in an irreducible way, then  $Q_\rho$ is the orbit of a period point of period $q$.
\end{theoc}

Finally, by combining Theorem C and results from Koropecki (see \cite{Koropecki}), Franks and Le Calvez (see \cite{franks/lecalvez:2003})  and Koropecki, Le Calvez and Nassiri (see \cite{KLN}), we are able to deduce the following extension of the above mentioned result by Le Calvez, by improving the smoothness requirements.

\begin{theod}\label{theod:patriceimprovement}
Let $f$ be an area-preserving homeomorphism of the closed annulus $\Ac=\T{1}\times [0,1]$ which is isotopic to the identity. Let $\hat{f}$ be a lift of $f$ to $\R\times [0,1]$. For every $\rho$ in $\rot(\hat{f})$ there exists a compact $f$-invariant set $Q_\rho$ such that for every point of $Q_\rho$ has a rotation number well-defined and it is equal to $\rho$.
\end{theod}

The paper draws heavily from Le Calvez's Brouwer Equivariant Theory and also from a forcing theory for surface homeomorphisms recently developed by Le Calvez and the second author. These results are predicated on the study of maximal isotopies, Brouwer-Le Calvez transverse foliations and transverse paths to these foliations, concepts that are better described in Section \ref{sec: preliminary results}. It is the use of this new tool, coupled with a careful analysis of possible transverse paths of maps of the annulus and classical ergodic theory lemmas that allows us to deduce the main results. As stated before, in Section \ref{sec: preliminary results} we introduce the basic lemmas and results from the above mentioned forcing theory, as well as detail the concept of rotation set for annular homeomorphisms. Section \ref{sec: proof of theorem A} is devoted to showing Theorem A. Section \ref{sec:proofoftheoremB} includes the proof of Theorem B and Section \ref{sec:example} provides an example displaying how tight are the hypotheses of Theorem B. Section \ref{sec:realizationresults} contains the proofs of Theorem C and Theorem D.

\subsection{Acknowledgements}
F. T. was partially supported by the Alexander von Humboldt foundation and by Fapesp and CNPq.  J. C. was supported by CNPq-Brasil and by FONDECYT postdoctoral grant N 3170455 entitled ``Algunos problemas para grupos de homeomorfismos de superficies''. The authors would like to thanks A. Koropecki for substantial discussions regarding this work in general and in particular about Theorem D.

\section{Preliminary results} \label{sec: preliminary results}
In this section, we state different results and definitions that will be useful in the rest of the article. The main tool will be the ``forcing theory'' introduced recently by Le Calvez and the second author (see \cite{lct} for more details) and further developed in \cite{newlct}. This theory will be expressed in terms of maximal isotopies, transverse foliations and transverse trajectories.

\subsection{Transverse paths to surface foliations}

Let $M$ be an oriented surface. An \textit{oriented singular foliation} $\mathcal{F}$ on $M$ is a closed set $\sing (\mathcal{F})$, called \textit{the set of singularities of $\mathcal{F}$}, together with an oriented foliation $\mathcal{F}'$ on the complement of $\sing (\mathcal{F})$, called \textit{the domain of $\mathcal{F}$} denoted by $\dom(\mathcal{F})$, i.e. $\mathcal{F}'$ is a partition of $\dom(\mathcal{F})$ into connected oriented $1$-manifolds (circles or lines) called \textit{leaves of $\mathcal{F}$}, such that for every $z$ in $\dom(\mathcal{F})$ there exist an open neighborhood $W$ of $z$, called \textit{trivializing neighborhood} and an orientation-preserving homeomorphism called \textit{trivialization chart at $z$}, $\fonc{h}{W}{(0,1)^2}$ that sends the restricted foliation $\mathcal{F}|_W$ onto the vertical foliation oriented downward. If the singular set of $\mathcal{F}$ is empty, we say that the foliation $\mathcal{F}$ is \textit{non singular}. For every $z\in \dom(\mathcal{F})$ we write $\phi_z$ for the leaf of $\mathcal{F}$ that contains $z$, $\phi^+_z$ for the positive half-leaf and $\phi^-_z$ for the negative one.

A \textit{path} on $M$ is a continuous map $\fonc{\gamma}{J}{M}$ defined on an interval $J$ of $\R$. In absence of ambiguity its image also will be called a path and denoted by $\gamma$. A path $\fonc{\gamma}{J}{\dom(\mathcal{F})}$ is \textit{positively transverse}\footnote{In the whole text ``transverse'' will mean ``positively transverse''} to $\mathcal{F}$ if for every $t_0\in J$ there exists a trivialization chart $h$ at $\gamma(t_0)$ such that the application $t\mapsto \pi_1(h(\gamma(t)))$, where $\fonc{\pi_1}{(0,1)^2}{(0,1)}$ is the projection on the first coordinate, is increasing in a neighborhood of $t_0$. We note that if $\widehat{M}$ is a covering space of $M$ and $\fonc{\widehat{\pi}}{\widehat{M}}{M}$ the covering projection, then $\mathcal{F}$ can be naturally lifted to a singular foliation $\widehat{\mathcal{F}}$ of $\widehat{M}$ such that $\dom(\widehat{\mathcal{F}})= \widehat{\pi}^{-1}(\dom(\mathcal{F}))$. We will denote $\widetilde{\dom}(\mathcal{F})$ the universal covering space of $\dom(\mathcal{F})$ and $\widetilde{\mathcal{F}}$ the foliation lifted from $\mathcal{F}|_{\dom(\mathcal{F})}$. We note that $\widetilde{\mathcal{F}}$ is a non singular foliation of $\widetilde{\dom}(\mathcal{F})$.  Moreover if $\fonc{\gamma}{J}{\dom(\mathcal{F})}$ is \textit{positively transverse} to $\mathcal{F}$, every lift $\fonc{\widehat{\gamma}}{J}{\dom(\widehat{\mathcal{F}})}$ of $\gamma$ is \textit{positively transverse} to $\widehat{\mathcal{F}}$. In particular every lift $\fonc{\widetilde{\gamma}}{J}{\widetilde{\dom}(\mathcal{F})}$ of $\gamma$ to the universal covering space $\widetilde{\dom}(\mathcal{F})$ of $\dom(\mathcal{F})$  is \textit{positively transverse} to the lifted non singular foliation $\widetilde{\mathcal{F}}$.

\subsubsection{Transverse paths intersecting $\mathcal{F}$-transversally}

 A \textit{line} on $M$ is an injective and proper path $\fonc{\lambda}{J}{M}$, that is, the interval $J$ is open and the pre-image of every compact subset of $M$ is compact. It inherits a natural orientation induced by the usual orientation of $\R$. Let $\lambda$ be a line of the plane $\R^2$. The complement of $\lambda$ has two connected component, $R(\lambda)$ which is on the right of $\lambda$ and $L(\lambda)$ which is on its left. We will say that a line $\lambda$ \textit{separates} $X$ from $Y$, if $X$ and $Y$ belong to different connected components of the complement of $\lambda$. Let us consider three pairwise disjoint lines $\lambda_0$, $\lambda_1$ and $\lambda_2$ in $\R^2$. We  say that $\lambda_2$ is \textit{above} $\lambda_1$ relative to $\lambda_0$ (and $\lambda_1$ is \textit{below} $\lambda_2$ relative to $\lambda_0$) if none of the lines separates the two others; and if $\gamma_1$ and $\gamma_2$ are two disjoint paths that join $z_1=\lambda_0(t_1)$, $z_2=\lambda_0(t_2)$ to $z'_1\in \lambda_1$, $z'_2\in \lambda_2$ respectively, and that do not meet the three lines but at the ends,
then $t_2>t_1$. This notion does not depend on the orientation of $\lambda_1$ and $\lambda_2$ but depends of the orientation of $\lambda_0$ (see Figure \ref{fig:orderoflines}).

\begin{center}
\begin{figure}[h!]
  \centering
    \includegraphics{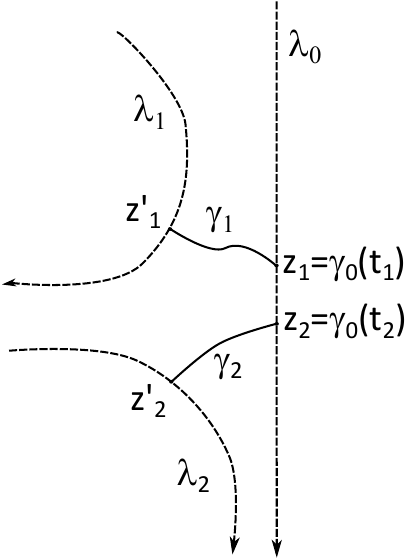}
  \caption{$\lambda_2$ is above $\lambda_1$ relative to $\lambda_0$.}
  \label{fig:orderoflines}
\end{figure}
\end{center}

Let suppose that $\mathcal{F}$ is an oriented singular foliation on an oriented surface $M$. Let $\fonc{\gamma_1}{J_1}{\dom(\mathcal{F})}$ and $\fonc{\gamma_2}{J_2}{\dom(\mathcal{F})}$ be two transverse paths. Suppose that there exist $t_1\in J_1$ and $t_2\in J_2$ such that  $\gamma_1(t_1)=\gamma_2(t_2)$. We say that $\gamma_1$ \textit{intersects $\gamma_2$  $\mathcal{F}$-transversally and positively at $\gamma_1(t_1)=\gamma_2(t_2)$}, if there exist $a_1,\,b_1$ in $J_1$ satisfying $a_1<t_1<b_1$, and $a_2,\, b_2$ in $J_2$ satisfying $a_2<t_2<b_2$ such that if $\fonc{\widetilde{\gamma}_1}{J_1}{\widetilde{\dom}(\mathcal{F})}$ and $\fonc{\widetilde{\gamma}_2}{J_2}{\widetilde{\dom}(\mathcal{F})}$ are lifts of $\gamma_1$ and $\gamma_2$ respectively, satisfying $\widetilde{\gamma}_1(t_1)=\widetilde{\gamma}_2(t_2)$ then
\begin{itemize}
  \item $\phi_{\widetilde{\gamma}_2(a_2)}$ is below $\phi_{\widetilde{\gamma}_1(a_1)}$ relative to $\phi_{\widetilde{\gamma}_1(t_1)}$; and
  \item $\phi_{\widetilde{\gamma}_2(b_2)}$ is above $\phi_{\widetilde{\gamma}_1(b_1)}$ relative to $\phi_{\widetilde{\gamma}_2(t_2)}$.
\end{itemize}
See Figure \ref{fig:intersectiontransverse}. In this situation we also say that $\gamma_2$ intersects $\gamma_1$ $\mathcal{F}$-transversally and negatively at $\gamma_1(t_1)=\gamma_2(t_2)$, and that $\gamma_1$ and $\gamma_2$ have a $\mathcal{F}$-transversal intersection at $\gamma_1(t_1)=\gamma_2(t_2)$.

\begin{center}
\begin{figure}[h!]
  \centering
   \includegraphics{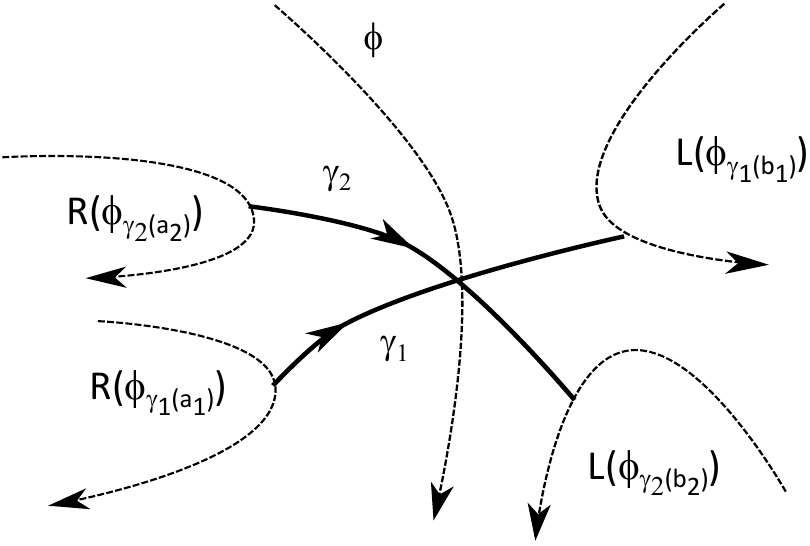}
  \caption{The paths $\gamma_1$ and $\gamma_2$ intersect $\mathcal{F}$-transversally and positively at $\gamma_1(t_1)=\gamma_2(t_2)$.}
  \label{fig:intersectiontransverse}
\end{figure}
\end{center}

If $\gamma_1=\gamma_2$, we will that $\gamma_1$ has a $\mathcal{F}$-self intersection. This means that if $\widetilde{\gamma}_1$ is a lift of $\gamma$ to the universal covering of $\dom(\mathcal{F})$, then there exists a covering automorphism $T$ such that $\widetilde{\gamma}_1$ and $T(\widetilde{\gamma}_1)$ have a $\widetilde{\mathcal{F}}$-transverse intersection at $\widetilde{\gamma}_1(t_1)=T(\widetilde{\gamma}_1(t_2))$.

\subsection{Maximal isotopies, transverse foliations, admissible paths}

\subsubsection{Isotopies, maximal isotopies}

Let $M$ be an oriented surface. Let $f$ be a homeomorphism of $M$. An \textit{identity isotopy} of $f$ is a path that joins the identity to $f$ in the space of homeomorphisms, furnished with the $C^0$-topology. We will say that \textit{$f$ is isotopic to the identity} if the set of identity isotopies of $f$ is not empty. Let $I=\left(f_t\right)_{t\in [0,1]}$ be an identity isotopy of $f$. Given $z\in M$ we can define the \textit{trajectory of $z$} as the path $I(z): t \mapsto f_t(z)$. For every integer $n\geq 1$ we define $I^n(z)=\prod_{0\leq k<n} I(f^k(z))$ by concatenation. Futhermore, we define

$$ I^{\N}(z)= \prod_{k\in \N} I(f^k(z)), \quad I^{-\N}(z)= \prod_{k\in \N} I(f^{-k}(z)), \quad I^{\Z}(z)= \prod_{k\in \Z} I(f^k(z)).     $$
The last path will be called the \textit{whole trajectory of $z$}. One can define the fixed point of $I$ as $\fix(I)= \cap_{t\in [0,1]} \fix(f_t)$, which is the set of point with trivial whole trajectory. The complement of $\fix(I)$ will called the \textit{domain of $I$}, and it will be denoted by $\dom(I)$.

In general, let us say that an identity isotopy of $f$ is a maximal isotopy, if there is no fixed point of $f$ whose trajectory is contractible relative to the fixed point set of $I$. A very recent result of F. B\'eguin, S. Crovisier and F. Le Roux (see \cite{BCLR16}) asserts that such an isotopy always exists if $f$ is isotopic to the identity (a slightly weaker result was previously proved by O. Jaulent (see \cite{jaulent}).

\begin{theo}[\cite{jaulent}, \cite{BCLR16}]\label{existence maximal isotopy}
  Let $M$ be an oriented surface. Let $f$ be a homeomorphism of $M$ which is isotopic to the identity and let $I'$ be an identity isotopy of $f$. Then there exists an identity isotopy $I$ of $f$ such that:
  \begin{itemize}
    \item[(i)] $\fix(I')\subset \fix(I)$;
    \item[(ii)] $I$ is homotopic to $I'$ relative of $\fix(I')$;
    \item[(iii)] there is no point $z\in \fix(f)\setminus \fix(I)$ whose trajectory $I(z)$ is homotopic to zero in $M\setminus \fix(I)$.
  \end{itemize}
\end{theo}
We will say that an identity isotopy $I$ satisfying the conclusion of Theorem \ref{existence maximal isotopy} is a \textit{maximal isotopy}. We note that the last condition of the above theorem can be stated in the following equivalent form:
 \begin{itemize}
    \item[(iii')] if $\widetilde{I}=(\widetilde{f}_t)_{t\in [0,1]}$ is the identity isotopy that lifts $I|_{M\setminus \fix(I)}$ to the universal covering space of $M\setminus \fix(I)$, then $\widetilde{f}_1$ is fixed point free.
  \end{itemize}

\subsubsection{Transversal foliations}

Let us recall the equivariant foliation version of the Plane Translation Theorem due to P. Le Calvez (see \cite{lecalvez1}).

\begin{theo}[\cite{lecalvez1}]\label{existence transverse foliation}
  Let $M$ be an oriented surface. Let $f$ be a homeomorphism of $M$ which is isotopic to the identity and let $I$ be a maximal identity isotopy of $f$. Then there exists an oriented singular foliation $\mathcal{F}$ with $\dom(\mathcal{F})=\dom(I)$, such that for every $z\in \dom(I)$ the trajectory $I(z)$ is homotopic, relative to the endpoints, to a positively transverse path to $\mathcal{F}$ and this path is unique defined up to equivalence.
\end{theo}
We will say that a foliation $\mathcal{F}$ satisfying the conclusion of Theorem \ref{existence transverse foliation} is \textit{transverse} to $I$. Observe that if $\widehat{M}$ is a covering space of $M$ and $\fonc{\widehat{\pi}}{\widehat{M}}{M}$ the covering projection, a foliation $\mathcal{F}$ transverse to a maximal identity isotopy $I$ lifts to a foliation $\widehat{\mathcal{F}}$ transverse to the lifted isotopy $\widehat{I}$. In particular, the foliation $\widetilde{\mathcal{F}}$ on $\widetilde{\dom(\mathcal{F})}$ is non singular which is transverse to the isotopy $\widetilde{I}$. This last property is equivalent to saying that every leaf $\widetilde{\phi}$ of $\widetilde{\mathcal{F}}$ is a Brouwer line of $\widetilde{f}$, that is $\widetilde{f}(\widetilde{\phi})\subset L(\widetilde{\phi})$ and $\widetilde{f}^{-1}(\widetilde{\phi})\subset R(\widetilde{\phi})$, where $ L(\widetilde{\phi})$ and $R(\widetilde{\phi})$ are the left and right connected components of the complement of $\widetilde \phi$, defined so that they are compatible with the orientation of the line.

Given $z\in M$ we will write $I_{\mathcal{F}}^1(z)$ for the class of paths that are positively transverse to $\mathcal{F}$, that join $z$ to $f(z)$ and that are homotopic in $\dom(\mathcal{F})$ to $I(z)$, relative to the endpoints. We will also use the notation $I_{\mathcal{F}}^1(z)$ for every path in this class and we will called it the transverse trajectory of $z$. More generally, for every integer $n\geq 1$ we can define $I^n_{\mathcal{F}}(z)=\prod_{0\leq k<n} I_{\mathcal{F}}^1(f^k(z))$ by concatenation, that is either a transverse path passing through the points $z$, $f(z)$, $\cdots$, $f^n(z)$, or a set of such paths. Futhermore, we define

$$ I^{\N}_{\mathcal{F}}(z)= \prod_{k\in \N} I_{\mathcal{F}}^1(f^k(z)), \quad I^{-\N}_{\mathcal{F}}(z)= \prod_{k\in \N} I_{\mathcal{F}}^1(f^{-k}(z)), \quad I^{\Z}_{\mathcal{F}}(z)= \prod_{k\in \Z} I_{\mathcal{F}}^1(f^k(z)).     $$
The last path will be called the \textit{whole transverse trajectory of $z$}.\\

 Let us state the following result that will be useful later.

\begin{lemma}[\cite{lct}]\label{lemma10LCT}
  Fix $z\in \dom(I)$, an integer $n\geq 1$, and parameterize $I^n_{\mathcal{F}}(z)$ by $[0,1]$. For every $0<a<b<1$, there exists a neighborhood $V$ of $z$ such that for every $z'$ in $V$, the path $I^n_{\mathcal{F}}(z)|_{[a,b]}$ is equivalent to a subpath of $I^n_{\mathcal{F}}(z')$. Moreover, there exists a neighborhood $W$ of $z$ such that for every $z'$ and $z''$ in $W$, the path $I^n_{\mathcal{F}}(z')$ is equivalent to a subpath of $I^{n+2}_{\mathcal{F}}(f^{-1}(z''))$.
\end{lemma}

\subsubsection{Admissible paths}

We will say that a transverse path $\fonc{\gamma}{[a,b]}{\dom(I)}$  is \textit{admissible of order $n$} ($n\geq 1$ is an integer) if it is equivalent to a path $I^{n}_{\mathcal{F}}(z)$, $z$ in $\dom(I)$. It means that if $\fonc{\widetilde{\gamma}}{[a,b]}{\widetilde{\dom}(I)}$ is a lift of $\gamma$, there exists a point $\widetilde{z}$ in $\widetilde{\dom}(I)$ such that $\widetilde{z}\in \phi_{\widetilde{\gamma}(a)}$ and $\widetilde{f}^n(\widetilde{z})\in \phi_{\widetilde{\gamma}(b)}$, or equivalently, that
$$    \widetilde{f}^n(\phi_{\widetilde{\gamma}(a)})\cap \phi_{\widetilde{\gamma}(b)} \neq \emptyset.  $$

The fundamental proposition (Proposition 20 from \cite{lct}) is a result about maximal isotopies and transverse foliations that permits us to construct new admissible paths from a pair of admissible paths.

\begin{prop}[\cite{lct}]\label{proposition20LCT}
  Suppose that $\fonc{\gamma_1}{[a_1,b_1]}{M}$ and $\fonc{\gamma_2}{[a_2,b_2]}{M}$ are two transverse paths that intersect $\mathcal{F}$-transversally at
  $\gamma_1(t_1)=\gamma_2(t_2)$. If $\gamma_1$ is admissible of order $n_1$ and $\gamma_2$ is admissible of order $n_2$, then the paths
  $\gamma_1|_{[a_1,t_1]}\gamma_2|_{[t_2,b_2]}$ and $\gamma_2|_{[a_2,t_2]}\gamma_1|_{[t_1,b_1]}$ are admissible of order $n_1+n_2$.
\end{prop}

One deduces immediately the following result.

\begin{lemma}[\cite{lct}]\label{corollary22LCT}
  Let $\fonc{\gamma_i}{[a_i,b_i]}{M}$, $1\leq i \leq r$, be a family of $r\geq 2$ transverse paths. We suppose that for every $i\in \{1,\cdots, r\}$ there exist $s_i\in [a_i,b_i]$ and $t_i\in [a_i,b_i]$ such that:
  \begin{itemize}
    \item[(i)] $\gamma_i|_{[s_i,b_i]}$ and $\gamma_{i+1}|_{[a_{i+1},t_{i+1}]}$ intersect $\mathcal{F}$-transversally at $\gamma_i(t_i)=\gamma_{i+1}(s_{i+1})$ if $i<r$;
    \item[(ii)] one has $s_1=a_1<t_1<b_1$, $a_r<s_r<t_r=b_r$ and $a_i<s_i<t_i<b_i$ if $1<i<r$;
    \item[(iii)] $\gamma_i$ is admissible of order $n_i$.
  \end{itemize}
Then $\prod_{1\leq i \leq r} \gamma_i|_{[s_i,t_i]}$ is admissible of order $\sum_{1\leq i \leq r} n_i$.
\end{lemma}

The following result is a consequence of Proposition 23 from \cite{lct}.

\begin{coro}[\cite{lct}]\label{corollary24LCT}
  Let $\fonc{\gamma}{[a,b]}{M}$ be a transverse path admissible of order $n$. Then there exists $\fonc{\gamma'}{[a,b]}{M}$ a transverse path, also admissible of order $n$, such that $\gamma'$ has no $\mathcal{F}$-transverse self-intersection and $\phi_{\gamma(a)}=\phi_{\gamma'(a)}$, $\phi_{\gamma(b)}=\phi_{\gamma'(b)}$.
\end{coro}

\subsection{Forcing theory results}

The presence of topological horseshoes have been a paradigmatic feature of dynamical systems, and its prevalence as a phenomena is well-known. In the article \cite{newlct}, which is a natural continuation of the article \cite{lct}, the authors develop a new criteria for the existence of topological horseshoes for surface homeomorphisms in the isotopy class of the identity based on the notions of maximal isotopies, transverse trajectories and traverse trajectories. The fundamental results of \cite{newlct} is that the existence of an admissible path with a $\mathcal{F}$-self-intersection implies the existence of a horseshoe.

Moreover, into of the proof of the main theorem from \cite{newlct}, one obtains the following results on the existence of elements of the rotation set or of compact sets with prescribed rotation numbers. In our setting, let consider $f$ a homeomorphism of the closed annulus $\Ac:=\T{1}\times [0,1]$ which is isotopic to the identity. Let $\widehat{\Ac}:=\R\times [0,1]$ be the universal covering of $\Ac$ and let $\widehat{f}$ be a lift of $f$ to $\widehat{\Ac}$. Let $I'$ be an identity isotopy of $f$, such that its lift to $\widehat{\Ac}$ is an identity isotopy of $\widehat{f}$. Let $I$ be a maximal identity isotopy of $f$ larger than $I'$ and let $\mathcal{F}$ be a singular foliation transverse to $I$. Let $\widehat{\mathcal{F}}$ be the lift of $\mathcal{F}$ to $\widehat{\Ac}$. The following result is a direct consequences of Theorem M of \cite{newlct}.

\begin{prop}
\label{pr:newperiodicforcing2}
   Suppose that there exists an admissible transverse path $\widehat{\gamma}:[a,b]\to \widehat{\Ac}$ of order $q\geq 1$ and an integer $p\in\Z$ such that $\widehat{\gamma}$ and $\widehat{\gamma}+(p,0)$ intersect $\widehat{\mathcal{F}}$-transversally at $\phi_{\widehat{\gamma}(t)}=\phi_{(\widehat{\gamma}+(p,0))(s)}$, with $a<s<t<b$. Then for any $0<\theta \leq p/q$, there exists a nonempty compact subset $Q_\theta$ of $\Ac$ such that, for each $z\in Q_\theta$, one has that $\rot(\widehat{f},z)=\theta$.
\end{prop}

\subsection{Rotation set of annular homeomorphisms}

\subsubsection{Rotation set}

We will denote by $\T{1}:=\R/\Z$ the circle and by $\Ac:=\T{1}\times [0,1]$ the closed annulus. We endow the annulus $\Ac$ with its usual topology and orientation. Let $\fonc{\widehat{\pi}}{\widehat{\Ac}=\R\times [0,1]}{\Ac}$ be the universal covering map of $\Ac$ defined by $\widehat{\pi}(x,y)=(x+\Z,y)$. Let $f$ be a homeomorphism of $\Ac$ that is isotopic to the identity (that is $f$ preserves the orientation and the boundary components) and let $\widehat{f}$ be a lift of $f$ to $\widehat{\Ac}$, i.e ($f\circ \widehat{\pi}=\widehat{\pi}\circ \widehat{f}$). We can define the \textit{displacement function} $\fonc{\rho_1}{\Ac}{\R}$ as
$$ \rho_1(x+\Z,y)=p_1(\widehat{f}(x,y))- x,    $$
where $\fonc{p_1}{\widehat{\Ac}}{\R}$ is the projection on the first coordinate and $(x,y)\in \widehat{\pi}^{-1}(x+\Z,y)$. Let $X\subset \Ac$ be a compact $f$-invariant set. We will denote by  $\mathcal{M}_f(\Ac,X)$ the set of all $f$-invariant Borel probability measures supported in $X$. If $\mu$ is in $\mathcal{M}_f(\Ac,X)$, we define its \textit{rotation number} as $$ \rot(\widehat{f},\mu):= \int_{\Ac} \rho_1\, d\mu.   $$
Then we define the \textit{rotation set of $\widehat{f}$ in $X$} as
 $$ \rot(\widehat{f},X):=\{\rot(\widehat{f},\mu): \mu \in \mathcal{M}_f(\Ac,X) \}.   $$
\begin{rema}
  For every $p\in\Z$ and every $q\in\Z$, the map $\widehat{f}^q+(p,0)$ is a lift of $f^q$ and we have  $\rot(\widehat{f}^q+(p,0),X)=q\rot(\widehat{f},X)+p$.
\end{rema}

For every measure $\mu$ in $\mathcal{M}_f(\Ac,X)$, by the Ergodic Decomposition Theorem, there is a unique Borel probability measure $\tau$ on $\mathcal{M}_f(\Ac,X)$ supported in $\mathcal{M}^{e}_f(\Ac,X)$ (the set of all ergodic measure in $\mathcal{M}_f(\Ac,X)$) such that for every continuous function $\fonc{\varphi}{{\Ac}}{\R}$ we have
$$\int_{\Ac} \varphi\,d\mu= \int_{\mathcal{M}^{e}_f(\Ac,X)} \left( \int_{\Ac} \varphi\,d\nu \right) \, d\tau(\nu). $$
Hence if $\rot(\widehat{f},\mu)$ is an endpoint of $\rot(\widehat{f},X)$, then $\tau$ almost all ergodic measures $\nu$ that appear in the ergodic decomposition of $\mu$ have a well defined rotation number, which is equal to $\rot(\widehat{f},\mu)$. Moreover, if $\nu$ in $\mathcal{M}_f(\Ac,X)$ is ergodic, then by Birkhoff Ergodic Theorem $\nu$-almost every point $z$ has a rotation number well-defined and it is equal to the rotation number of $\nu$, i.e. for a lift $\hat{z}$ of $z$ the limit of the sequence $\left((p_1(\widehat{f}^n(\widehat{z}))-p_1(\widehat{z}))/n\right)_{n\in\N}$ exists which will be denoted by $\rot(\widehat{f},z)$, and we have

$$ \rot(\widehat{f},z)= \lim_{n\to +\infty} \frac{p_1(\widehat{f}^n(\widehat{z}))-p_1(\widehat{z})}{n}= \lim_{n \to +\infty} \frac{1}{n} \sum_{k=0}^{n-1} \rho_1(f^k(z))=\rot(\widehat{f},\nu). $$

If $X=\Ac$, we denoted the rotation set of $\widehat{f}$ by $\rot(\widehat{f})$ instead $\rot(\widehat{f},\Ac)$. We have the following theorem, which can be deduced from \cite{franksannals1}.

\begin{theo}[Frank's Theorem]\label{frankstheorem}
  Let $\fonc{f}{\Ac}{\Ac}$ be a homeomorphism of $\Ac$ which is isotopic to the identity. Suppose that $f$ preserves a Borel probability measure of full support. Let $\widehat{f}$ be a lift of $f$ to $\R\times [0,1]$. Then for every rational number $r/s$, written in an irreducible way, in the interior of the rotation set of $\widehat{f}$ there exists a point $\widehat{z}\in \R\times [0,1]$ such that $\widehat{f}^s(\widehat{z})=\widehat{z}+(r,0)$.
\end{theo}

\subsubsection{Dynamics near a boundary component with positive rotation number}

We recall the local dynamics of a homeomorphism of the close annulus near to a boundary component of $\Ac$ with positive rotation number. Let us consider a homeomorphism $f$ of the closed annulus which is isopotic to the identity. Let $\widehat{f}$ be a lift of $f$ to $\R\times [0,1]$. We will suppose that the rotation number of the lower boundary component of $\Ac$,
$$\rho_0:=\lim_{n\to +\infty}\frac{p_1(\widehat{f}^n(x,0))-x}{n},$$
where $x\in\R$, is positive. Therefore $\widehat f|_{\R\times \{0\}}$ has no fixed point, and so we can consider
$$ m:= \inf_{\widehat{z} \in \R\times \{0\}  } (p_1(\widehat{f}(\widehat{z}))-p_1(\widehat{z}))>0.$$
We deduce the next result.

\begin{lemma}\label{lm: lemma29}
  There exists a real number $\delta>0$ such that for every $\widehat{z}\in \R\times [0,\delta]$ we have
   $$ \frac{m}{2}< p_1(\widehat{f}(\widehat{z}))-p_1(\widehat{z}).  $$
\end{lemma}
We deduce the following result.
\begin{coro}\label{Co: corollary210}
  For every real number $M>0$, there exist a real number $\delta>0$ and an integer $n\geq 1$ such that for every $\hat{z}\in \R\times [0,\delta]$ we have
   $$ M< p_1(\widehat{f}^n(\widehat{z}))-p_1(\widehat{z}).  $$
\end{coro}

\subsection{Dynamics of an oriented foliation in a neighborhood of an isolated singularity}

 In this subsection, we consider an oriented singular foliation $\mathcal{F}$ on an oriented surface $M$ which has an isolated singulary $z_0$. A \textit{sink} (resp. a \textit{source}) of $\mathcal{F}$ is an isolated singularity point $z_0$ of $\mathcal{F}$ such that there is a homeomorphism $h$ from a neighborhood $U$ of $z_0$ to the open unit disk $\D$ of $\R^2$ which sends $z_0$ to $0\in \D$ and sends the restricted foliation $\mathcal{F}\vert_{U\setminus\{z_0\}}$ to the radial foliation on $\D\setminus\{0\}$ with the leaves towards (resp. backwards) to $0$. We recall that for every $z\in \dom(\mathcal{F})$ we will write $\phi_z$ for the leaf of $\mathcal{F}$ that contains $z$, $\phi^+_z$ for the positive half-leaf and $\phi^-_z$ for the negative one.

We can define the $\alpha$-limit and $\omega$-limit sets of each leaf $\phi$ of $\mathcal{F}$ as follows:
$$ \alpha(\phi):= \bigcap_{z\in \phi} \overline{ \phi^-_z}, \quad \text{ and } \quad \omega(\phi):= \bigcap_{z\in \phi} \overline{ \phi^+_z}.$$

We will use the following result due to Le Roux that describes the dynamics of an oriented singular foliation $\mathcal{F}$ near an isolated singularity (see \cite{ler1}). For our purpose we state a simplified version of him result.

\begin{prop}[\cite{ler1}]\label{propofoliationnearsingularity}
  Let $\mathcal{F}$ be an oriented singular foliation on an oriented surface $M$. Let $z_0$ be an isolated singularity of $\mathcal{F}$. Then at least one of the following cases holds:
  \begin{itemize}
    \item[(1)] every neighborhood of $z_0$ contains a closed leaf of $\mathcal{F}$;
    \item[(2)]  there exist a leaf of $\mathcal{F}$ whose $\alpha$-limit set is reduced to $z_0$ and a leaf of $\mathcal{F}$ whose $\omega$-limit set is reduced to $z_0$; or
    \item[(3)]  $z_0$ is either a sink or a source of $\mathcal{F}$.
  \end{itemize}
\end{prop}

We will use also the following result due to Le Calvez. He proved that, in this case, $\mathcal{F}$ is also \textit{locally transverse} to $I$ at $z_0$ (see \cite{lec3}), that is, for every neighborhood $V_{z_0}$ of $z_0$ there exists a neighborhood $W_{z_0}$ of $z_0$ contained in $V_{z_0}$ such that for every $z\in W_{z_0}$, $z\neq z_0$, its transverse trajectory $I_{\mathcal{F}}^1(z)$ is contained in $V_{z_0}\setminus\{z_0\}$.

 Let $z_0$ be a point in an oriented surface $M$ and let $f$ be a homeomorphism of $M$ which fixes $x_0$. An \textit{local identity isotopy} of $f$ is a path that joins the identity to $f$ in the space of homeomorphisms of $M$ fixing $z_0$, furnished with the $C^0$-topology.

\begin{prop}[\cite{lec3}]\label{propolocallytransverse}
  Let $I$ be a local identity isotopy of a homeomorphism of an oriented surface $M$. Suppose that $\mathcal{F}$ is an oriented singular on $M$ which is transverse to $I$. If $M\setminus\{z_0\}$ is not a topological sphere, then $\mathcal{F}$ is also \textit{locally transverse} to $I$ at $z_0$.
\end{prop}

\subsection{Periodic disks for area-preserving homeomorphisms of $\Ac$}

In this subsection, let $f:\Ac\to\Ac$ be an area-preserving homeomorphism, and let $\hat f$ be a lift of $f$ to the universal covering space.  A set $O\subset \Ac$ is a topological disk if it is homeomorphic to an open disk of $\R^2$. We need the following result which can be deduced from Theorem 4 of \cite{LecalvezKoroTal}.

\begin{lemma}
Let $g:\R^2\to\R^2$ be a homeomorphism and let $z\in \R^2$ be such that, for every neighborhood $U$ of $z$, there exists three disjoint and invariant topological disks $O_1, O_2, O_3$ such that $O_i\cap U\not=\emptyset,\,i\in\{1, 2, 3\}$. If each $O_i$ does not contain wandering points, then $g(z)=z$.
\end{lemma}

As a consequence we obtain.

\begin{lemma}\label{lm:unbounded_disk}
Let $O\subset \Ac$ be a topological disk and $\widehat O$ a connected component of $\widehat{\pi}^{-1}(O)$. If $\widehat O$ is not bounded, and if  there exists integers $p, q$ with $q>0$ such that $\widehat f^{q}(\widehat O)=\widehat O +(p,0)$, then there exists $\widehat x$ such that $\hat f^{q}(\widehat x)=\widehat x+(p,0)$ and such that $\widehat{\pi}(\widehat x)\in \partial O$.
\end{lemma}
\begin{proof}
Note that, since $O$ is a topological disk, the sets $\widehat O_i= \widehat O+(i,0)$, with $i\in\Z$ are all disjoint, and since $\widehat O$ is unbounded one may find a sequence of points $(\widehat{z}_i)_{i\in\Z}$ of points in $\widehat O_i$ that accumulates on a point $\widehat{x}$. Note that each $\widehat O_i$ is invariant by $\widehat g= \widehat f^q-(p,0)$. Furthermore, since each $\widehat O_i$ projects to a topological disk and is invariant by $\widehat g$, the dynamics of $\widehat g$ restricted to each $\widehat O_i$ is conjugated to the dynamics of $f^q$ restricted to $O$. Since $f^q$ is area-preserving, it has no wandering points and therefore $\widehat g$ has no wandering points in each $\widehat O_i$. The result follows from the previous lemma.
\end{proof}

\subsection{Essential sets on the annulus, prime ends and realization of rotation vectors in continua}

We say that an open subset $O$ of $\A=\T1\times(0,1)$ is {\it essential} if it contains a simple closed curve which is not homotopic to a point. If $O$ is open, connected and essential, then the {\it filling of $O$}, that is the union of $O$ and all the compact connected components of its complement, is a topological open annulus homeomorphic to $\A$. We say that $K\subset \A$  is an \textit{essential continuum} if it is a continuum (i.e. connected and compact) which separates the two ends of $\A$. Likewise, if $K\subset \Ac=\T1\times[0,1]$, then we say that $K$ is an essential continuum if it is contained in $\A$ and is an {\it essential continuum} for $\A$. If $K\subset\Ac$, we denote by $U_+=U_+(K)$ and $U_-=U_-(K)$ be the components of $\Ac\setminus K$ containing $\T{1}\times \{1\}$ and $\T{1}\times \{0\}$ respectively.\\

\subsubsection{Prime ends rotation numbers} We start recalling a very brief description of prime ends rotation numbers (for a more complete description see \cite{KLN}). Let $f$ be a homeomorphism of $\Ac$, and let $K$ be an essential $f$-invariant continuum. Collapsing the lower and upper boundary components of $\Ac$ to points $S$ and $N$, respectively, we obtains a topological $2$-dimensional sphere and the dynamics induced by $f$ fixes these two points. We consider $U_+$ and $U_-$ as defined above. The sets $U_+^*=U_+ \cup \{N\}$ and $U_-^*=U_- \cup \{S\}$ are invariant open topological disks. It is known that one may define a prime end compactification $\widetilde{U}_+^*$  (respectively $\widetilde{U}_-^*$) of $U_+^*$ (resp. $U_-^*$) which, as a set, is the disjoint union of $U_+^*$ (resp. $U_-^*$ ) with a topological circle, called the circle of prime ends. This compactification can be endowed with a suitable topology making it homeomorphic to the closed unit disk $\overline{\D}$ of the plane. Furthermore, and more relevantly, this compactification is such that the restriction of $f$ to $U_{+}$ (resp $U_-$) extends in a unique way to a homeomorphism of $\widetilde{U}_+^*$ (resp. $\widetilde{U}_-^*$).

Lifting the inclusion $U_+ \to  U_+^*\setminus \{N\} $ to the universal covering, we obtain a map $\pi_+: \widehat{\pi}^{-1}(U_+) \to H_+:=\{(x,y)\in\R^2: y>0\}$ and a homeomorphism $\widehat{f}_+: \overline{H}_+ \to \overline{H}_+$ such that $ \pi_+ \widehat{f}_+ \vert_{\widehat{\pi}^{-1}(U_+)}= \widehat{f}_+ \pi_+$ and $\widehat{f}_+(x+1,y)= \widehat{f}_+(x,y)+(1,0)$. Similarly, we obtain map $\pi_-: \widehat{\pi}^{-1}(U_-) \to H_-:=\{(x,y)\in\R^2: y<0\}$ and a homeomorphism $\widehat{f}_-: \overline{H}_- \to \overline{H}_-$ such that $ \pi_- \widehat{f}_- \vert_{\widehat{\pi}^{-1}(U_-)}= \widehat{f}_- \pi_-$ and $\widehat{f}_-(x+1,y)= \widehat{f}_-(x,y)+(1,0)$. The \textit{upper} (respectively \textit{lower}) \textit{prime end rotation number} of $K$ associated to $\widehat{f}$ is defined as
$$  \rho^{\pm}(\widehat{f},K):= \lim_{n\to +\infty} \frac{p_1(\widehat{f}_\pm^n(x,0))-x}{n},$$
which is independent of $x$. If $ \rho^{+}(\widehat{f},K)= \rho^{-}(\widehat{f},K)$ we call this number \textit{the prime end rotation number of $K$}. We note that for every integers $p,q$, the map $\widehat{f}^q+(p,0)$ is a lift of $f^q$ and we have $$\rho^{\pm}(\widehat{f}^q+(p,0),K)= q \rho^{\pm}(\widehat{f},K)+p.$$

If $f$ is a homeomorphism of $\A$ and if $O\subset \A$ is a pre-compact essential open annulus which is $f$-invariant, one can likewise define the prime ends compactification of $O$ in the following way. Let $S, N$ be the two ends of $\A$. Since $O$ is an essential open annulus, its complement has exactly two connected components. Let $K_{N}$ be the subset of the boundary of $O$ that is contained in the connected component that is a neighborhood of $N$, and let $K_S=\partial O\setminus K_N$. One notes that both $K_N$ and $K_S$ are $f$-invariant essential continua and that $O\subset U_{-}(K_N)\cap U_{+}(K_S)$. The prime ends compactifictaion $O^*$ of $O$ is the disjoint union of $O$ with two topological circles $C_{S}$ and $C_{N}$ with an appropiate topology such that there exists neighborhoods $V_S, V_N $ of $C_S$ and $C_N$ respectively in $O^*$, a neighborhood $W_S$ of the prime ends circle in $\widetilde{U}_+^*(K_S)$  and $W_N$ neighborhood of the prime ends circle in $\widetilde{U}_-^*(K_N)$, such that $V_S$ is homeomorphic to $W_S$ and such that $V_N$ is homeomorphic to $W_N$. Done in this way, $O^*$ is homeomorphic to $\Ac$, and $f$ has a unique continuous extension $f^*$ that is a homeomorphism of $O^*$. One can then verify that, if $\widehat{f^*}$ is a lift of $f^{*}$ to the universal covering, then the rotation number of the restriction of $f^{*}$ to $C_N$ is the same as $\rho^{-}(\widehat{f}, K_N)$ and that the rotation number of the restriction of $f^{*}$ to $C_S$ is the same as $\rho^{+}(\widehat{f}, K_S)$.

\subsubsection{Realization of rotation vectors in continua}

We will need the following result, which can be derived from \cite{Koropecki} and \cite{KLN}.

\begin{prop}\label{pr:boundaryannulus}
Let $V\subset\A$ be an essential open annulus, $K'$ be a connected component of $\partial V$, and $K$ be the union of $K'$ and all the pre-compact connected components of its complement. Let $f:\A\to\A$ be an area-preserving homeomorphism such that $f(V)=V$ and $\widehat{f}$ a lift of $f$ to its universal covering. Then there exists $\rho$ such that every point in $K$ has the same rotation number $\rho$. Furthermore, $\rho$ is the prime ends rotation number of $K$.
\end{prop}
\begin{proof}
The same proof as Theorem 2.8 of \cite{Koropecki} shows that the rotation number of any point in $K'$ is the same, and that it is precisely the prime end rotation number of $K'$. It remains to show the same is true for any point in $K$. Suppose, for a contradiction, that there are points with two different rotation numbers, $\rho_{-}$ and $\rho_{+}$ in $K$. By Proposition 5.4 of \cite{franks/lecalvez:2003} one has that for every rational $p/q$ in $(\rho_{-},\rho_{+})$ there exists a point $z_{p/q}$ in $K$ such that, if $\widehat{z_{p/q}}\in \widehat{\pi}^{-1}(z_{p/q})$, then $\widehat{f}^q(\widehat{z_{p/q}})=\widehat{z_{p/q}}+(p,0)$. By Theorem $A$ of \cite{Koropecki}, since the rotation interval of the restriction of $f$ to $K$ is a non-degenerate closed interval, one deduces that there are two ergodic measures, $\mu_1$ and $\mu_2$, supported in $K$, such that both $\rot(\widehat f,\mu_1)$ and $\rot(\widehat f,\mu_2)$ are irrational numbers and such that $\rot(\widehat f,\mu_1)\not= \rot(\widehat f,\mu_2)$. This implies that there exists recurrent points $z_1$ and $z_2$ in $K$ such that the rotation number of $z_1$ is $\rot(\widehat f,\mu_1)$ and the rotation number of $z_2$ is $\rot(\widehat f,\mu_2)$. But if $O$ is a pre-compact connected component of the complement of $K'$, then $O$ is an open topological disk in $\A$, and since $f$ preserves area there exists some integer $q_0$ such that $f^{q_0}(O)=O$. This implies that every recurrent point of $O$ that has a rotation number must have a rational rotation number. Therefore neither $z_1$ nor $z_2$ can lie in pre-compact connected components of the complement of $K'$, and so both points belong to $K'$. But this is a contradiction, since every point in $K'$ has the same rotation number.
\end{proof}

\begin{lemma}\label{lm:rotationnumber of the limit}
Let $f:\A\to\A$ be a homeomorphism, $\widehat f$ be a lift of $f$ and assume that there exists $\rho$ a real number and $K$ an $f$-invariant compact set such that, for every $f$-invariant ergodic measure $\nu$ supported on $K$, the rotation number of $\nu$ is $\rho$. Then for every $\varepsilon>0$ there exists $N_0=N_0(\varepsilon)$ such that for all $\widehat z\in \widehat{\pi}^{-1}(K)$ and all $n\ge N_0$, $\vert p_1(\widehat f^{n}(\widehat z))-p_1(\widehat z)-n\rho\vert <n\varepsilon/2$. Furthermore, for every $\varepsilon>0$ there exists $\delta>0$ such that if $y$ is a point whose whole orbit lies in the $\delta$-neighborhood of $K$ and has rotation number, then $\vert \rot(\widehat f, y)-\rho\vert\le\varepsilon.$
\end{lemma}
\begin{proof}
Suppose, by contradiction, there exists a sequence of points $\widehat{z_k}\in \widehat{\pi}^{-1}(K)$ and an increasing sequence of integers $n_k$ such that $\vert p_1(\widehat f^{n_k}(\widehat z_k))-p_1(\widehat z_k)-n_k\rho\vert \ge n_k\varepsilon/2$. Let $\nu_k$ be the measure
$\frac{1}{n_k}\sum_{i=0}^{n_k-1}\delta_{f^{i}(\widehat{\pi}(\widehat{z}_k))}$, and we assume that $\nu_k$ converges in the weak-$*$ topology to a measure $\overline{\nu}$, otherwise we take a subsequence. One verifies that $\overline{\nu}$ is an $f$-invariant measure supported on $K$ and that $\vert\rot(\widehat f, \overline{\nu})-\rho\vert\ge \varepsilon/2$. By the Ergodic Decomposition Theorem, we obtain that there exists an $f$-invariant ergodic measure $\nu$ supported on $K$ whose rotation number is not $\rho$. This is a contradiction. For the second assertion, given $\varepsilon>0$, if $N_0=N_0(\varepsilon)$, note that the continuous function $g:\A\to\R, \, g(z)=\frac{1}{N_0}\left[ p_1(\widehat f^{N_0}(\widehat z))-p_1(\widehat z)\right]-\rho$, where $\widehat z$ is any point in $\widehat{\pi}^{-1}(z)$, takes values in $(-\varepsilon/2, \varepsilon/2)$ for $z\in K$. Therefore, there exists some $\delta>0$ such that, for every $y$ in a $\delta$-neighborhood of $K$, the function $g$ takes values in $[-\varepsilon,\varepsilon]$. This implies that, if the whole orbit of $y$ lies in the $\delta$-neighborhood of $K$, then
$$\left\vert\frac{1}{iN_0}\left[ p_1(\widehat f^{iN_0}(\widehat y))-p_1(\widehat y)\right]-\rho\right\vert=\left\vert\frac{1}{i}\sum_{j=0}^{i-1}g(f^{jN_0}(y))\right\vert\le \varepsilon$$
where $\widehat{y}\in\pi^{-1}(y)$. Therefore, if $y$ has a rotation number well-defined, then $\vert \rot(\widehat f, y)-\rho\vert\le\varepsilon$.
\end{proof}

\begin{prop}\label{pr:continuumwithinessentialinterior}
Let $f:\A\to\A$ be an area-preserving homeomorphism, let $K'$ be an $f$-invariant essential continuum and let $K$ be the union of $K'$ and all the pre-compact connected components of its complement. If the interior of $K$ is inessential, then every point in $K$ has the same rotation number.
\end{prop}
\begin{proof}
Since $K'$ is essential, so is $K$. By Proposition \ref{pr:boundaryannulus}, one knows that there exist $\rho_1$ and $\rho_2$ such that every point in $\partial U_{-}(K)$ has rotation number $\rho_1$, and that any point in $\partial U_{+}(K)$ has rotation number $\rho_2$. If the interior of $K$ is inessential, we have that $\partial U_{-}(K)$ and $\partial U_{+}(K)$ are not disjoint, and so $\rho_1$ and $\rho_2$ are the same. Note that $\partial K=\partial U_{-}(K)\cup\partial U_{+}(K)$. Let $O$ be a connected component of the interior of $K$, and note that again $O$ is an $f$-periodic open topological disk in $\A$, and there exist integers $p_0, q_0$ with $q_0$ positive, such that if $\widehat O$ is a connected component of $\widehat{\pi}^{-1}(O)$, then $\widehat f^{q_0}(\widehat O)= \widehat O+ (p_0,0)$. We claim that $\rho_1=p_0/q_0$. Indeed, if $\widehat O$ is bounded, then every point in the closure of $O$ has the same rotation number and it is $p_0/q_0$, and since $\partial O\subset \partial K$ we get the claim, and if $\widehat O$ is unbounded, then by Lemma \ref{lm:unbounded_disk} there exists a point $\bar z\in \partial O$ such that $f^{q_0}(\bar z)=\bar z$ and such that the rotation number of $\bar z$ is $p_0/q_0$. Since every point in $\partial O\subset \partial K$ has the same rotation number, we deduce that  $\rho_1=p_0/q_0$. Since $O$ was arbitrary, one deduces that any recurrent point in $K$ has rotation number $\rho_1$, and therefore any ergodic measure supported in $K$ has rotation number $\rho_1$. But if a compact invariant set is such that every ergodic measure supported on it has the same rotation number, this implies that every point in the set has a well defined rotation number and that this number is $\rho_1$.
 \end{proof}

\subsection{Regions of instability}

Let $A$ be an open topological annulus, denote by $S, N$ the two topological ends of $A$ and let $f:A\to A$ be a homeomorphism  preserving both the orientation and the ends of $A$. There are two classical definitions of regions of instability: we say that $A$ is a {\it Birkhoff region of instability} if for any $U, V$, neighborhoods of $S$ and $N$ respectively, there exist $n_1, n_2>0$ such that $f^{n_1}(U)\cap V\not=\emptyset$ and such that $f^{n_2}(V)\cap U\not=\emptyset$. We note that, if the dynamics of $f$ is such that every point is non-wandering, then equivalently $A$ is a Birkhoff region of instability if for any $U, V$, neighborhoods of $S$ and $N$ respectively, the full orbit of $U$ intersects $V$. This implies that,  if the dynamics of $f$ is non-wandering and $A$ is not a Birkhoff region of instability, then there exists $O_{S}$, $O_{N}$, neighborhoods of $S$ and $N$ respectively, which are $f$-invariant and disjoint.

We say that $A$ is a {\it Mather region of instability} if there exist points $z_1, z_2$ in $A$ such that both $f^{-n}(z_1)$ and $f^{n}(z_2)$ converge to $S$ when $n$ goes to infinity, and such that both $f^{n}(z_1)$ and $f^{-n}(z_2)$ converge to $N$ when $n$ goes to infinity. One clearly has that a Mather region of instability is also a Birkhoff region of instability.  Let us introduce a new definition, which is stronger than the first one but weaker than the second. We say that $A$ is a {\it $SN$ mixed region of instability} if, for every neighborhood $U$ of $S$, one can find points $z_1, z_2$ in $U$ such that both $f^{n}(z_1)$ and $f^{-n}(z_2)$ converge to $N$ when $n$ goes to infinity. One defines similarly a $NS$ mixed region of instability if for every neighborhood $V$ of $N$, one can find points $z_1, z_2$ in $V$ such that both $f^{n}(z_1)$ and $f^{-n}(z_2)$ converge to $S$ when $n$ goes to infinity.

A general question which is unknown is to give conditions such that any Birkhoff region of instability must also be a Mather region of instability. This was shown to hold for twists maps, and also generically for area-preserving homeomorphisms (\cite{Mather-regionsofinstavility} and \cite{franks/lecalvez:2003} respectively). Our next proposition shows that, in the area-preserving context, that usually if there exists an annulus that is a Birkhoff region of instability, one can always find a subannulus that is a $SN$ mixed region of instability.

\begin{prop}\label{pr:semi-mather}
Let $f:\Ac\to\Ac$ be a homeomorphism which is isotopic to the identity that preserves a measure of full support such that its interior is a Birkhoff region of instability for the restriction of $f$. Assume further that the rotation set of the restriction of $f$ to $\Ac$ is not a single point.  Then there exists some $A_2\subset \A$, which is also an open topological annulus, such that $A_2$ is a $SN$ mixed region of instability and such that the rotation set of the restriction of $f$ to $A_2^*$ is the same as that of $f$.
\end{prop}
(See Figure \ref{fig:Proofprop218} Left).

\begin{center}
\begin{figure}[h!]
  \centering
    \includegraphics{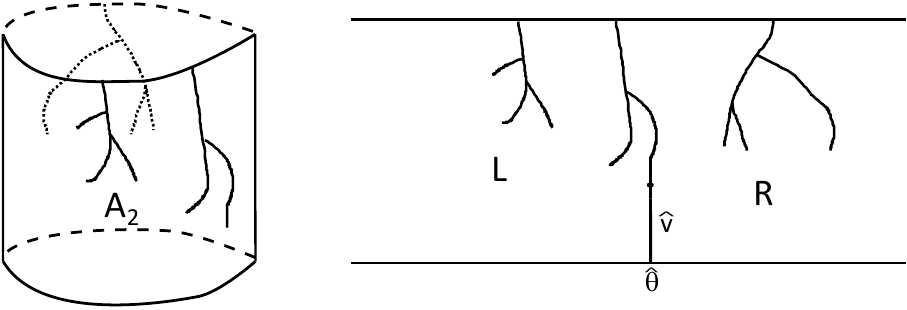}
  \caption{Left side: the annulus $A_2$ in the statement of Proposition \ref{pr:semi-mather}  Right side: The segment $\widehat{v}$ on the proof of Lemma \ref{lm:Lemma219}.}
  \label{fig:Proofprop218}
\end{figure}
\end{center}

Before we begin with the proof of this proposition, let us repeat a construction that dates back to Birkhoff itself, and has more recently been used in the study of homeomorphisms of the annulus and of the $2$-dimensional torus in \cite{transitiveannulusI, transitivetorus, boundedunbounded}. We refer the reader to Section 4.3 of \cite{boundedunbounded} for more details of the constructions. For the remainder of this section we assume that we are under the hypothesis of Proposition \ref{pr:semi-mather}.

For every $0<\varepsilon<1$, let $B_{\varepsilon}(f)$ be the connected component of the set
$$\bigcap_{i\in\N}f^{-i}\left(\T{1}\times[\varepsilon, 1] \right)$$
that contains $\T{1}\times\{1\}$, and we omit the dependence on $f$ whenever the context permit. Its a classical result that $B_{\varepsilon}$ intersects the circle $\T{1}\times\{\varepsilon\}$ and is forward invariant. We also define the set $\omega_{\varepsilon}(f)$ as the connected component of
$$\omega_{\varepsilon}=\bigcap_{i\in\Z}f^{-i}\left(\T{1}\times[\varepsilon, 1] \right)$$
that contains $\T{1}\times\{1\}$, and again we omit the dependence on $f$ whenever the context permit. Let us point remark that $\omega_{\varepsilon}(f)=\omega_{\varepsilon}(f^{-1})$. One verifies trivially that $\omega_{\varepsilon}$ is closed, invariant, and that the $\omega$-limit set of any point in $B_{\varepsilon}$ is contained is $\omega_{\varepsilon}$. Furthermore, since we are assuming that  $\A$ is a Birkhoff region of instability, one has that for each $\varepsilon>0$, the interior of $B_{\varepsilon}$ does not contain an essential annulus.

One also verifies that if $\varepsilon_1>\varepsilon_2$, then $\omega_{\varepsilon_1}\subset\omega_{\varepsilon_2}$. On the other hand, if in this situation $\omega_{\varepsilon_2}\subset \T{1}\times[\varepsilon_1,1]$, then $\omega_{\varepsilon_2}=\omega_{\varepsilon_1}$.

\begin{lemma}\label{lm:Lemma219}
There exists $\varepsilon_0$ such that, if $\varepsilon<\varepsilon_0$, then $\omega_{\varepsilon}=\omega_{\varepsilon_0}$.
\end{lemma}
\begin{proof}
First note that, for every positive integer $p$, $\omega_{\varepsilon}(f)\subset\omega_{\varepsilon}(f^p)$. Let $g$ be a power of $f$ and $\widehat{g}$ be a lift of $g$ such that the $\rot(\widehat g)$ contains the interval $[-1,1]$, and such that the rotation number of the restriction of $g$ to both boundaries does not belong to $\{-1,0,1\}$. There exists some $\varepsilon_0$ such that, if $\hat{z}\in\R\times\{0\}$, then the ball with radius $\varepsilon_0$ and center $\widehat z$ is free for $\widehat g$. Suppose, for a contradiction that there exists some $\varepsilon_1<\varepsilon_0$ such that $\omega_{\varepsilon_1}=\omega_{\varepsilon_1}(g)$ is not equal to $\omega_{\varepsilon_0}=\omega_{\varepsilon_0}(g)$. If this is true, and if $$\delta=\min_{s\in [0,1]}\{\exists \theta\in\T{1}, (\theta,s)\in \omega_{\varepsilon_1}\},$$
then $\varepsilon_1\le\delta\le\varepsilon_0$, because otherwise one would have that $\omega_{\varepsilon_1}$ is $g$-invariant and contained in $\bigcap_{i\in\Z}g^{-i}\left(\T{1}\times[\varepsilon_0, 1] \right)$, and therefore must be a subset of  $\omega_{\varepsilon_0}$. There exists some $\theta\in\T{1}\times\{0\}$ such that $(\theta,\delta)\in\omega_{\varepsilon_1}$. Let $v=\theta\times[0,\delta)$ be a line segment. Note that $v$ is disjoint from $\omega_{\varepsilon_1}$.

Let now $\widehat v$ be a connected component of $\widehat{\pi}^{-1}(v)$, that contains a point $(\widehat{\theta},0)$ with $0\le \widehat{\theta}<1$, let $\widehat{\omega_{\varepsilon_1}}=\widehat{\pi}^{-1}(\omega_{\varepsilon_1})$ and let $F=\widehat v\cup \widehat{\omega_{\varepsilon_1}}$. First note that the complement of $\widehat{\omega_{\varepsilon_1}}$ has a connected component $A$ that contains the strip $\R\times [0,\delta)$. We claim that $A$ is the unique connected component of the complement of $\widehat{\omega_{\varepsilon_1}}$. Indeed, the complement of $A$ is invariant and contained in $\R\times[\varepsilon_1,1]$. Furthermore, since any connected component of $\widehat{\pi}(A)^C$ has a point of $\omega_{\varepsilon_1}$ and the later is connected, we have that $\widehat{\pi}(A)^C$ is connected, and therefore it is contained in $\omega_{\varepsilon_1}$.

The complement of $F$ can have at most two connected components, $L$, which contains $(-\infty,\widehat{\theta})\times\{0\}$, and $R$, which contains $(\widehat{\theta},\infty)\times\{0\}$. Let us show that these are different connected components. (See Figure \ref{fig:Proofprop218} Right). If not, there would be an arc $\gamma$ joining $(-1, 0)$ and $(1,0)$ entirely contained in $F^C$. If $\beta=\gamma\cup\left([-1,1]\times\{0\}\right)$, then $\beta$ is the image of a simple closed curve,  which is disjoint from $\widehat{\omega_{\varepsilon_1}}$ but such that the point $(\widehat{\theta},\delta)$ is in a different connected component from the complement of $\beta$ than $\R\times\{1\}$. This contradicts the fact that of $\omega_{\varepsilon_1}$ is connected and also contains $\T{1}\times\{1\}$.

Note also that, if $\hat z$ is a point in the complement of $\widehat{\omega_{\varepsilon_1}}$, one can find an arc $\alpha$ in $(\widehat{\omega_{\varepsilon_1}})^C$ joining $\hat z$ to a point $(a,0)$. Since $\widehat{\omega_{\varepsilon_1}}$ is invariant by integer horizontal translations, one gets that if $p>\abs{a}+1$ is sufficiently large, $\alpha+(p,0)$ and $\alpha-(p,0)$ are both disjoint from  $\widehat{v}$ and $\widehat{\omega_{\varepsilon_1}}$. One checks that $\alpha+(p,0)\subset R$, since it contains $(a+p,0)$ and that $\alpha-(p,0)\subset L$, since it contains $(a-p,0)$. We get that, for each $\widehat z\notin \widehat{\omega_{\varepsilon_1}}$, there exists a sufficiently large $p$ such that $\widehat{z}+(p,0)$ is in $R$ and $\widehat{z}-(p,0)$ is in $L$.

There are two cases to consider. Either the rotation number of $\T{1}\times\{0\}$ for $\widehat{g}$ is positive or it is negative. The rest of the proof is similar in both situations, so we will assume that it is positive. This implies that $\widehat{g}\left((\widehat{\theta},0)\right)$ belongs to $R$ and $\widehat{g}^{-1}\left((\widehat{\theta},0)\right)$ belongs to $L$. Since $\widehat{v}$ is disjoint from $\widehat{\omega_{\varepsilon_1}}$, we get that the image of $\widehat{v}$ by $\widehat{g}$ is contained in $R$, and that its preimage is contained in $L$. This implies that $\widehat{g}(R)$ does not intersect $\widehat{v}$. Since $\widehat{g}(R)$ is also disjoint from $\widehat{\omega_{\varepsilon_1}}$ and is connected, one has that it must be contained in a connected component of the complement of $F$, and as the image by $\widehat{g}$ of $(\widehat{\theta}, \infty)\times\{0\}$ intersects itself, we deduce that $\widehat{g}(R)\subset R$.

Finally, since $g$ preserves a measure of full support, by the hypothesis on the rotation set of $\widehat{g}$, one can find  there exists some $\bar{z}$ in $\Ac$ such that, if $\widehat{\bar{z}}$ is a lift of $\bar{z}$, then $\widehat{g}(\widehat{\bar{z}})=\widehat{\bar{z}}-(1,0)$. By Proposition \ref{pr:continuumwithinessentialinterior}, one knows that every point $\omega_{\varepsilon_1}$ has the same rotation number for $\widehat{g}$, and since $\T{1}\times\{1\}$ belongs to $\omega_{\varepsilon_1}$, this number is not $-1$. This implies that $\bar{z}$ does not belong to $\omega_{\varepsilon_1}$, and therefore there exists some integer $p$ such that $\widehat{\bar{z}}+(p,0)$ belongs to $R$, and such that $\widehat{g}^{2p}(\widehat{\bar{z}}+(p,0))=\widehat{\bar{z}}-(p,0)$ belongs to $L$, which is a contradiction since that $R$ is positively invariant. This shows that $\omega_{\varepsilon}(g)\subset\omega_{\varepsilon_0}(g)$ for all $\varepsilon<\varepsilon_0$.

Now, if $\varepsilon<\varepsilon_0$, since $\omega_{\varepsilon}(f)\subset \omega_{\varepsilon}(g)$, we get that $\omega_{\varepsilon}(f)\subset \omega_{\varepsilon_0}(g)\subset \T{1}\times [\varepsilon_0, 1]$, and since $\omega_{\varepsilon}(f)$ is $f$-invariant, connected and contains $\T{1}\times \{1\}$, we deduce that it is contained in $\omega_{\varepsilon_0}(f)$. Since it holds that $\omega_{\varepsilon_0}(f)\subset\omega_{\varepsilon}(f)$, we have the result.
\end{proof}

\textit{End of the proof of Proposition \ref{pr:semi-mather}.} Let $A_2$ be the complement of $\omega_{\varepsilon_0}\cup (\T{1}\times \{0\})$. Note that $A_2$ is open, contains $\T{1}\times (0, \varepsilon_0)$ and therefore separates $\T{1}\times \{0\}$ and $\T{1}\times \{1\}$, and its complement has exactly two connected components. Therefore $A_2$ is an essential open topological annulus. Let $S$ be the end of $A_2$ corresponding to $\T{1}\times \{0\}$, and let $N$ be the other end.

If $U$ is a neighborhood in $A_2$ of $S$, then there exists $\varepsilon>0$ such that $\T{1}\times (0, 2\varepsilon)$ is contained in $U$. As noted before, $B_{\varepsilon}(f)$ has a point in $z_1$ in $\T{1}\times \{\varepsilon\}  \subset U$ and the $\omega$-limit set of any point in $B_{\varepsilon}(f)$ is contained in $\omega_{\varepsilon}=\omega_{\varepsilon_0}$, one has that the future orbit of $z_1$ in $U$ converges to $N$. Likewise, one knows that $\omega_{\varepsilon}(f^{-1})=\omega_{\varepsilon}(f)=\omega_{\varepsilon_0}(f)$, and since $B_{\varepsilon}(f^{-1})$ has a point $z_2$ in $U$ whose $\omega$-limit set for $f^{-1}$, and therefore whose $\alpha$-limit set for $f$, is contained in $\omega_{\varepsilon_0}(f)$. This shows that $A_2$ is a $SN$ mixed region of instability.

Finally, since the rotation set of $f$ is not a single point, it is a non-empty interval and since $f$ preserves a measure of full support, one knows that for every $p/q$ in the rotation set of $f$ there exists a periodic point $z_{p/q}$ such that the rotation number of $z_{p/q}$ is $p/q$. Since every point in $\omega_{\varepsilon_0}$ has the same rotation number which is equal to the rotation number of the restriction of $f$ to $\T{1}\times \{1\}$, and since every point in $\T{1}\times \{0\}$ has the same rotation number, we get that for all but possibly two values of $p/q$ in the rotation set of $f$, $z_{p/q}$ must belong to $A_2$. Since the rotation set of the restriction of $f$ to $A_2^*$ must be closed, we deduce it must be equal to the full rotation set of $f$.

\section{Proof of Theorem A} \label{sec: proof of theorem A}

In this section, we prove Theorem A. Let $f$ be a homeomorphism of the closed annulus $\Ac=\T{1}\times [0,1]$ which is isotopic to the identity and preserves a Borel probability measure $\mu$ with full support. Let $\widehat{f}$ be a lift of $f$ to $\R\times [0,1]$. Replacing $f$ by a power $f^q$ and the lift $\widehat{f}$ by a lift $\widehat{f}^q+(p,0)$, one can suppose that $\rot(\widehat{f})=[\alpha,\beta]$ with $\alpha<0<1<\beta$. We can also assume that the rotation numbers of both boundary components of $\Ac$ are  not equal to $1$. Suppose by contradiction that $\rot(\widehat{f},\mu)=\alpha$. We will fix a maximal identity isotopy $I$ of $f$ that lifts to a maximal identity isotopy of $\widehat{f}$, and we also fix an oriented singular foliation $\mathcal{F}$ transverse to $I$ and its lift $\widehat{\mathcal{F}}$.

\subsection{Preliminaries}

We recall first Atkinson's Lemma on Ergodic Theory (see \cite{atkinson}).
\begin{prop}\label{lemmaatkinson}
  Let $(Y,\mathcal{B},\nu)$ be a probability space, and let $\fonc{T}{Y}{Y}$ be an ergodic automorphism. If $\fonc{\varphi}{Y}{\R}$ is an integrable map such that $\int_Y \varphi\, d\nu=0$, then for every $B\in\mathcal{B}$ and every real number $\epsilon>0$, one has
  $$  \nu \left( \left\{ y\in B, \; \exists n\geq 0,\; T^n(y)\in B \text{ and } \abs{\sum_{k=0}^{n-1} \varphi(T^k(y))} <\epsilon \right\} \right)=\nu(B). $$
\end{prop}

As a corollary we have the following result (see Corollary 4.6 of \cite{boundedunbounded}).
\begin{coro}\label{cr:atkinsonourcase}
Let $\nu$ be an ergodic invariant measure for $f$ and $\varphi:\Ac\to \R$ be an integrable map such that $\int_{\Ac}\varphi\, d\nu=0$. Then for $\nu$-almost every point $z\in\Ac$, there exists an increasing sequence $q_l\to +\infty$ such that $f^{q_l}(z)\to z$ and $\sum_{j=0}^{q_l-1}\varphi(f^{j}(z))\to 0$.
\end{coro}

We also have the following result which can be derived directly from Proposition 4.3 of \cite{boundedunbounded}.
\begin{lemma}\label{le:atkinsonourcase2}
For every borelian $B\subset \Ac$ such that $\mu(B)>0$, there exists an $f$-invariant ergodic measure $\nu$ such that $\rot(\widehat{f},\nu)=\alpha$ and such that $\nu(B)>0$.
\end{lemma}

Finally, we deduce the following proposition.
\begin{prop}\label{consequences}
  There exists a set $X_\alpha$ in $\Ac$ with full $\mu$ measure such that, for every $z\in X_\alpha$ we have
\begin{itemize}
  \item[(i)] $z$ is a bi-recurrent point of $f$;
  \item[(ii)] the rotation number of $z$ is well-defined and $\rot(\widehat{f},z)=\alpha$; and
  \item[(iii)] one can find a sequence $(p_l,q_l)_{l\in\N}$ in $(-\N)\times \N$ such that, if $\widehat{z}$ belongs to $\widehat \pi^{-1}(z)$:
  $$ \lim_{l\to +\infty}q_l=+\infty,  \quad \lim_{l\to +\infty} (p_l-\alpha q_l)  = 0, \quad \lim_{l\to +\infty} \widehat{f}^{q_l}(\widehat{z})-\widehat{z}-(p_l,0)= 0 $$
  and a sequence $(p'_l,q'_l)_{l\in\N}$ in $\N\times \N$ satisfying:
  $$\lim_{l\to +\infty}q'_l=+\infty,  \quad \lim_{l\to +\infty} (p'_l+\alpha q'_l)  = 0, \quad \lim_{l\to +\infty} \widehat{f}^{-q'_l}(\widehat{z})-\widehat{z}+(p'_l,0)= 0.  $$
\end{itemize}
\end{prop}
\begin{proof}
Let $B$ be the complement of the set of points that satisfy properties (i) to (iii), and assume for a contradiction that $\mu(B)>0$. By Lemma \ref{le:atkinsonourcase2}, we can find an $f$-invariant ergodic measure $\nu$ such that $\rot(\widehat f,\nu)=\alpha$ and such that $\nu(B)>0$. Since $\nu$ is ergodic, $\nu$-almost every point in $B$ is bi-recurrent and has rotation number equal to $\alpha$. Applying Corollary \ref{cr:atkinsonourcase} using $\varphi(z)= p_1(\widehat{f}(\widehat{z})-\widehat{z})-\alpha$, one has that there exists a sequence of integers $q_l\to +\infty$ such that $f^{q_l}(z)\to z$ and $\sum_{j=0}^{q_l-1}\varphi(f^{j}(z))=p_1(\widehat{f}^{q_l}(\widehat{z})-\widehat{z})-q_i\alpha\to 0$.  Since $f^{q_l}(z)\to z$ one deduces there exists a sequence of integers $p_l$ such that $\widehat{f}^{q_l}(\widehat{z})-\widehat{z}-(p_l,0)\to(0,0)$ and so $p_l-q_l\alpha\to 0$. So $\nu$-almost every point in $B$ satisfy the first part of property (iii). A similar argument, using Corollary \ref{cr:atkinsonourcase} with $f^{-1}$ in place of $f$ and $\varphi'(z)= p_1(\widehat{f}^{-1}(\widehat{z})-\widehat{z})+\alpha$, gives us that $\nu$-almost every point in $B$ satisfy the second part of property (iii). Therefore $\nu$-almost every point of $B$ satisfy properties (i)-(iii), a contradiction.
\end{proof}

We have the following result, whose proof is immediate.
\begin{lemma} \label{consequences2}
 Let $\alpha<0$ and $r$ be a real number, and assume there exists a sequence $(p_l,q_l)_{l\in\N}$ in $(-\N)\times \N$ satisfying:
  $$ \lim_{l\to +\infty}q_l=+\infty,  \quad \lim_{l\to +\infty} (p_l-\alpha q_l)  = 0 .$$
  If $k>-\alpha r$, then there exists some integer  $l(k)\in \N $ such that for every integer $l\geq l(k)$, we have
   $$ \frac{p_l-k}{q_l+r}< \alpha. $$
\end{lemma}

\subsection{Proof of Theorem A}

Note that $\rot(\widehat{f})=[\alpha,\beta]$ and $\alpha<0<1<\beta$. We recall that we are assuming that $f$ preserves a Borel probability measure of full support, so by Frank's Theorem (Theorem \ref{frankstheorem}) one can find a fixed point $z$ of $f$ in the interior of the annulus such that, if $\widehat z$ is a lift of $z$, then $\widehat{f}(\widehat{z})=\widehat z +(1,0)$. Let $I^{\Z}_{\mathcal{F}}(z)$ be the whole transverse trajectory of $z=\widehat{\pi}(\widehat{z})$. We start by recalling some facts about the transverse path $I^{\Z}_{\mathcal{F}}(z)$. By definition, we have $I^{\Z}_{\mathcal{F}}(z)(0)=I^{\Z}_{\mathcal{F}}(z)(1)$. Let $\Gamma'$ be the loop naturally defined by the closed path $I^{\Z}_{\mathcal{F}}(z)|_{[0,1]}$. We know that every leaf, $\phi$, that meet $\Gamma'$ is wandering, that is, for every point $z\in \phi$ has a trivialization neighborhood $W$ such that each leaf of $\mathcal{F}$ contains no more than one leaf of $\mathcal{F}\vert_{W}$ (see \cite{lct} for more details). Consequently, if $t$ and $t'$ are sufficiently close, one has $\phi_{\Gamma'(t)}\neq \phi_{\Gamma'(t')}$. Moreover, because $\Gamma'$ is positively transverse to $\mathcal{F}$, one cannot find an increasing sequence $(a_n)_{n\in\N}$ and a decreasing sequence $(b_n)_{n\in\N}$, such that $\phi_{\Gamma'(a_n)}=\phi_{\Gamma'(b_n)}$. So, there exist real numbers $a,b$ with $0\leq a<b\leq 1$ such that $t\mapsto \phi_{I^{\Z}_{\mathcal{F}}(z)(t)}$ is injective on $[a,b)$ and satisfies $\phi_{I^{\Z}_{\mathcal{F}}(z)(a)}=\phi_{I^{\Z}_{\mathcal{F}}(z)(b)}$. Replacing $I^{\Z}_{\mathcal{F}}(z)$ by an equivalent transverse path, one can suppose that $I^{\Z}_{\mathcal{F}}(z)(a)=I^{\Z}_{\mathcal{F}}(z)(b)$. Let $\Gamma$ be the loop naturally defined by the closed path $I^{\Z}_{\mathcal{F}}(z)|_{[a,b]}$. The set $U_{\Gamma}=\bigcup_{t\in [a,b] } \phi_{I^{\Z}_{\mathcal{F}}(z)(t)}$ is an open annulus and $\Gamma$ is a simple loop. As $z$ is a periodic point we have the following result. This lemma is contained in the proof of Proposition 2 from \cite{lct}.

\begin{lemma}[\cite{lct}]\label{lemmaproofProposition2}
  Suppose that there exists $t<a$ such that $I^{\Z}_{\mathcal{F}}(z)(t)\notin U_{\Gamma}$. Then there exists $t'\in \R$ with $b<t'$ such that $I^{\Z}_{\mathcal{F}}(z)(t)$ and $I^{\Z}_{\mathcal{F}}(z)(t')$ are in the same connected component of the complement of $U_{\Gamma}$. Moreover $I^{\Z}_{\mathcal{F}}(z)|_{[t,t']}$ has a $\mathcal{F}$-transverse self-intersection.
\end{lemma}
\begin{proof}
  See the proof of Proposition 2 from \cite{lct}.
\end{proof}

Therefore, as $z$ is a periodic point, there are two possibilities for the whole transverse trajectory of $z$, $I^{\Z}_{\mathcal{F}}(z)$, namely:
\begin{itemize}
\item[I)] $I^{1}_{\mathcal{F}}(z)$ has a $\mathcal{F}$-transverse self-intersection; or
\item[II)] $I^{\Z}_{\mathcal{F}}(z)$  is equivalent to the natural lift of a simple loop $\Gamma$ of $\A$.
\end{itemize}
We will analyze each case separately.

\subsubsection{Case I}

Note that we can collapse the two boundary components of $\Ac$ to two points $S, N$ to obtain a sphere, and that there exist natural extensions of $f$ and $I$ to this sphere, as well as a new transversal foliation that has singularities in $S$ and $N$, which for simplicity we denote by $f^*, I^*$ and $\mathcal{F}^*$.  Since for this induced map $\Gamma'^*= I^{*1}_{\mathcal{F}^*}(z)$ is a loop homologous to zero  with a $\mathcal{F}^*$-transverse self-intersection, one can apply Proposition 7 of \cite{lct} and deduce that $I^{*2}_{\mathcal{F}^*}(z)$, which is admissible of order $2$, has a $\mathcal{F}^*$-transverse self-intersection. It follows that, $I^{2}_{\mathcal{F}}(z)$, has a $\mathcal{F}$-transverse self-intersection and that, if $\widehat{z}$ is a lift of $z$, one can find an integer $p_0$ such that the transverse path $\widehat{\gamma}:=I^{2}_{\widehat{\mathcal{F}}}(\widehat{z})$ has a $\widehat{\mathcal{F}}$-transverse intersection with $\widehat{\gamma}+(p_0,0)$. Lemma \ref{lemma10LCT} provides us a sufficiently small neighborhood $W$ of $\widehat{z}$ such that, for every $\widehat y \in W$, one has that
\begin{itemize}
\item[(a)] $I^{4}_{\widehat{\mathcal{F}}}(\widehat{f}^{-3}(\widehat y))$ contains a subpath equivalent to $I^{2}_{\widehat{\mathcal{F}}}(\widehat f^{-2}(\widehat z))= \widehat{\gamma}-(2,0)$, and
\item[(b)] $I^{4}_{\widehat{\mathcal{F}}}(\widehat f^{-1}(\widehat{y}))$ contains a subpath equivalent to $\widehat{\gamma}$.
\end{itemize}
Let $k>0$ be an integer such that $k-p_0>-2\alpha$. Fix $W_k\subset W$ be a neighborhood of $\widehat z$ such that $\widehat f^{k}(W_k)\subset W+(k,0)$, and let $\widehat{z_*}$ be a point in $W_k\cap \widehat{\pi}^{-1}(X_{\alpha})$, where $X_{\alpha}$ is the set provides by Proposition \ref{consequences}. By Proposition \ref{consequences} and Lemma \ref{consequences2}, one finds integers $p,\, q$, with $q>k+8$ sufficiently large, such that
\begin{itemize}
\item $\widehat f^{q}(\widehat{z_*})$ belongs to $W+(p,0)$;
\item $\frac{p+p_0-k}{q+1}<\alpha$.
\end{itemize}

Now, if $\widehat{y_*}=\widehat{f}^{k-1}(\widehat{z_*})$, then $\widehat{f}(\widehat{y_*})\in W+(k,0)$, which implies that $I^{4}_{\widehat{\mathcal{F}}}(\widehat{y_*})$ has $\gamma_1=\widehat{\gamma}+(k,0)$ as a subpath. Furthermore, as $\widehat f^{q}(\widehat{z_*})$ belongs to $W+(p,0)$, the transverse path  $I^{4}_{\widehat{\mathcal{F}}}(\widehat f^{q-3}(\widehat{z_*}))$ has $\gamma_2=\widehat{\gamma}+(p-2,0)$ as a subpath. Note that the path $\gamma^*=I^{q+1}_{\widehat{\mathcal{F}}}(\widehat{z_*})$ is admissible of order $q+1$, and $\gamma_2$ has a $\widehat{\mathcal{F}}$-transverse intersection with $\gamma_2+(p_0,0)=\gamma_1+(p+p_0-k-2,0)$. Therefore the path $\gamma^*\vert_{[q-3,q+1]}=I^{4}_{\widehat{\mathcal{F}}}(\widehat f^{q-3}(\widehat{z_*}))$ has a $\widehat{\mathcal{F}}$-transverse intersection with $\left(\gamma^* +(p+p_0-k-2,0)\right)\vert_{[k-1,k+3]}= I^{4}_{\widehat{\mathcal{F}}}(\widehat{y_*})+(p+p_0-k-2,0)$. By Proposition \ref{pr:newperiodicforcing2}, $\frac{p+p_0-k-2}{q+1}$ is in the rotation set of $\widehat{f}$, which is a contradiction.

\subsubsection{Case II}

Assume now that $I^{\Z}_{\mathcal{F}}(z)$ is equivalent to the natural lift of a simple loop $\Gamma$ in $\A$. Let consider $\fonc{\gamma}{\R}{\A}$ the natural lift of $\Gamma$ such that $\gamma(t+1)=\gamma(t)$ for every $t\in \R$ and $\widehat{\gamma}$ the lift of $\gamma$ to $\widehat{\Ac}$. One gets that the set of leafs intersecting $\Gamma$ is an open topological sub-annulus $U_{\Gamma}$ of $\A$, and since $\widehat{f}(\widehat{z})=\widehat{z}+(1,0)$, one gets that  $\widehat{U_{\Gamma}}=\widehat{\pi}^{-1}(U_{\Gamma})$ has a single connected component, and that $U_{\Gamma}$ is an essential annulus.

\begin{prop}\label{propo existence good transverse path}
There exist an admissible transverse path $\gamma^*$ and real numbers $a<a'<b'<b$ such that, if $\widehat{\gamma^*}$ is a lift of $\gamma^*$ to $\widehat{\Ac}$, then
 \begin{itemize}
    \item $\widehat{\gamma^*}\vert_{[a',b']}$ is equivalent to $\widehat{\gamma}\vert_{[s,s+1]}$ for some $s\in \R$ and $\widehat{\gamma^*}(a')=\widehat{\gamma^*}(b')$;
    \item $\widehat{\gamma^*}|_{(a,a')}$ is included in $\widehat{U_{\Gamma}}$ but it do not meet $\phi_{\widehat{\gamma}(s)}-(1,0)$ and $\widehat{\gamma^*}|_{(b',b)}$ is included in $\widehat{U_{\Gamma}}$ but it do not meet $\phi_{\widehat{\gamma}(s+1)}+(1,0)$; and
    \item $\widehat{\gamma^*}(a)$ and $\widehat{\gamma^*}(b)$ belong to the same connected component of the complement to $\widehat{U_{\Gamma}}$.
  \end{itemize}
\end{prop}
(See Figure \ref{fig:goodtransversepath}).

\begin{center}
\begin{figure}[h!]
  \centering
    \includegraphics{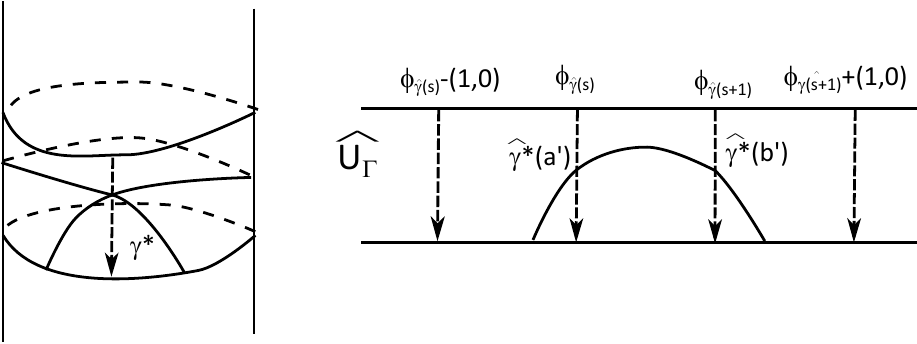}
  \caption{Left side: The transverse path $\gamma^*$. Right side: A lift of $\gamma^*$.}
  \label{fig:goodtransversepath}
\end{figure}
\end{center}

\begin{proof}[Proof of Proposition \ref{propo existence good transverse path}]
By density of the set $X_\alpha$, provides by Proposition \ref{consequences}, and Lemma \ref{lemma10LCT}, we can suppose that $\widehat{\gamma}\vert_{[0,2]}$ is equivalent to a subpath of $I_{\widehat{\mathcal{F}}}^\Z(\widehat{z_0})$, the whole transverse trajectory of a point $\widehat{z_0}$ that lifts a point $z_0$ in $X_\alpha$. We will denote $\widehat{\gamma_0}=I_{\widehat{\mathcal{F}}}^\Z(\widehat{z_0})$. Since the point $z_0$ has a negative rotation number, the transverse path $\widehat{\gamma_0}$ cannot be contained in $\widehat{U_\Gamma}$. Hence one can find real numbers
  $$t_1<t'_1< t'_2<t_2$$
  and integer $j^-$ and $j^+$ uniquely determined such that
\begin{itemize}
  \item $\widehat{\gamma_0}|_{[t'_1,t'_2]}$ is equivalent to $\widehat{\gamma}\vert_{[j^-,j^+]}$;
    \item $\widehat{\gamma_0}|_{[t_1,t'_1]}$ and $\widehat{\gamma_0}|_{[t'_2,t_2]}$ are included in $\widehat{U_\Gamma}$ but do not meet $\phi_{\widehat \gamma(0)}+(j^--1,0)$ and $\phi_{\widehat \gamma(0)}+(j^++1,0)$ respectively;
   \item $\widehat{\gamma_0}(t_1)$ and $\widehat{\gamma_0}(t_2)$ do not belong to $\widehat{U_\Gamma}$.
\end{itemize}
We claim that, if $\widehat{\gamma_0}(t_1)$ and $\widehat{\gamma_0}(t_2)$ belong to the same connected component of the complement of $\widehat{U_\Gamma}$ then, we are done. Indeed, we have that $j^+-j^-\geq 2$, we know that $ \widehat{\gamma_0}|_{[t_1,t'_2]}+ (j^+-1-j^-,0)$ and $ \widehat{\gamma_0}|_{[t'_1,t_2]}$ intersect $\widehat{\mathcal{F}}$-transversally at $\widehat{\gamma_0}(t)+ (j^+-1-j^-,0)=\widehat{\gamma_0}(t'_2)$. By Proposition \ref{proposition20LCT} the transverse path
$$  \widehat{\gamma}:= (\widehat{\gamma_0} + (j^+-1-j^-,0) )|_{[t_1,t]}\widehat{\gamma_0} |_{[t'_2,t_2]} $$
is a subpath of an admissible transverse path $\widehat{\gamma^*}$ which satisfies the proposition. In that follows we assume that
\begin{itemize}
  \item $\widehat{\gamma_0}(t_1)$ and $\widehat{\gamma_0}(t_2)$ belong to different connected components of the complement of $\widehat{U_\Gamma}$.
\end{itemize}
Since the point $z_0$ is bi-recurrent and the complement of the annulus $\widehat{U_\Gamma}$ saturated, that is, it is the union of singular point and leaves of $\widehat{\mathcal{F}}$, one can find real numbers $$t_2\leq t_3<t_4$$ such that
\begin{itemize}
  \item $\widehat{\gamma_0}(t_4)$ belongs to the same connected component of the complement of $\widehat{U_\Gamma}$ than $\widehat{\gamma_0}(t_1)$;
    \item $\widehat{\gamma_0}|_{[t_2,t_4)}$ does not meet this connected component of the complement of $\widehat{U_\Gamma}$;
   \item $\widehat{\gamma_0}|_{(t_3,t_4)}$ is included in $\widehat{U_{\Gamma}}$; and
   \item $\widehat{\gamma_0}(t_3)$ belongs to the same connected component of the complement of $\widehat{U_\Gamma}$ than $\widehat{\gamma_0}(t_2)$.
\end{itemize}
Observe now that, by Lemma \ref{lemmaproofProposition2}, there exists a non-zero integer $j$, such that $\widehat{\gamma}|_{[t_1,t_2]}$ and $\widehat{\gamma}|_{[t_3,t_4]}+(j,0)$ intersect $\widehat{\mathcal{F}}$-transversally at $\widehat{\gamma}(s)=\widehat{\gamma}(t)$ with $t_1<s<t_2<t_3<t<t_4$. Hence by Proposition \ref{proposition20LCT}, one has that the path
$$ \widehat{\gamma''_0}:= \widehat{\gamma}|_{[t_1,s]} (\widehat{\gamma}|_{[t,t_4]}+(j,0)),$$
is a subpath of an admissible transverse path. We can construct an admissible transverse path $\widehat{\gamma^*}$ as above.\\
\end{proof}

\subsection{End of the proof of Theorem A}

Let consider a lift $\fonc{\widehat{\gamma^*}}{[a,b]}{\widehat{\Ac}}$ of $\gamma^*$, provides by Proposition \ref{propo existence good transverse path}, to $\widehat{\Ac}$ and let $\widehat{\gamma}$ be a lift of the natural lift of $\Gamma$ such that $\widehat{\gamma}(0)=\widehat{\gamma}^*(a')$ and $\widehat{\gamma}(1)=\widehat{\gamma}^*(b')$. We suppose that $\widehat{\gamma^*}$ is admissible of order $N\geq 1$. We will denote by $\widehat{\phi_0}$ the leaf $\phi_{\widehat{\gamma}(0)}=\phi_{\widehat{\gamma^*}(a')}$ of $\widehat{\mathcal{F}}$.
Let $\widehat{U_{\Gamma}}$ be a lift of $U_{\Gamma}$ which is homeomorphic to $\R^2$. We have the following corollary to Proposition \ref{propo existence good transverse path} .

\begin{coro}
  For every integer $k\geq 1$, there exist an admissible of order $kN$ transverse path $\fonc{\widehat{\gamma_k}}{[a,b]}{\widehat{\Ac}}$, real numbers $a<a'_k<b'_k<b$ such that
  \begin{itemize}
    \item $\widehat{\gamma_k}|_{[a'_k,b'_k]}$ is equivalent to $\widehat{\gamma}|_{[0,k]}$;
    \item $\widehat{\gamma_k}|_{(a,a'_k)}$ is included in $\widehat{U_{\Gamma}}$ but it do not meet $\widehat{\phi}_0-(1,0)$ and $\widehat{\gamma}_k|_{(b'_k,b)}$ is included in $\widehat{U_{\Gamma}}$ but it do not meet $\widehat{\phi_0}+(k+1,0)$;
    \item $\widehat{\gamma_k}(a)$ and $\widehat{\gamma_k}(b)$ belong to the same connected component of the complement to $\widehat{U_{\Gamma}}$.
  \end{itemize}
\end{coro}
\begin{proof}
  Note that the paths $\widehat{\gamma^*}|_{[a',b]}$ and $(\widehat{\gamma^*}+(1,0))|_{[a,b']}$ intersect $\widehat{\mathcal{F}}$-transversally at $\widehat{\gamma^*}(b')=\widehat{\gamma^*}(a')+(1,0)$. It follows that for every integer $i\in \{1,\cdots, k-1\}$ the paths $(\widehat{\gamma^*}+(i,0))|_{[a',b]}$ and $(\widehat{\gamma^*}+(i+1,0))|_{[a,b']}$ intersect $\widehat{\mathcal{F}}$-transversally at $\widehat{\gamma^*}(b')+(i,0)=\widehat{\gamma^*}(a')+(i+1,0)$. Applying Lemma \ref{corollary22LCT}, one knows that the path
  $$ \widehat{\gamma_k}:= \widehat{\gamma^*}|_{[a,b']}\left(\prod_{i=1}^{k-2}(\widehat{\gamma^*}+(i,0))|_{[a',b']}\right) (\widehat{\gamma^*}|_{[a',b]}+(k-1,0))$$ is admisible of order $kN$. This completes the proof of the corollary.
\end{proof}

For every integer $k$ large enough consider real numbers $a''_k$ and $b''_k$ with $a'_k<a''_k<b''_k<b'_k$ such that $\phi_{\widehat{\gamma_k}(a''_k)}=\widehat{\phi_0}+(1,0)$ and  $\phi_{\widehat{\gamma_k}(b''_k)}=\widehat{\phi_0}+(k-1,0)$. Let $\widehat{z_k}$ be a point in $\dom(\widehat{I})$ such that $\widehat{\gamma_k}$ is a subpath of the path $I_{\widehat{\mathcal{F}}}^{kN}(\widehat{z_k})$, and consider the smaller integer $i_k$ and the larger integer $N_k$ in $\{1,\cdots, kN\}$ such that
  \begin{itemize}
    \item $\widehat{\gamma_k}|_{[a,a''_k]}$ is a subpath of $I_{\widehat{\mathcal{F}}}^{i_k}(\widehat{z_k})$;
    \item $\widehat{\gamma_k}|_{[b''_k,b]}$ is a subpath of $I_{\widehat{\mathcal{F}}}^{kN-i_k-N_k}(\widehat{f}^{i_k+N_k}(\widehat{z_k}))$.
  \end{itemize}
Let put $\widehat{y_k}= \widehat{f}^{i_k}(\widehat{z_k})$ and $\widehat{f}^{N_k}(\widehat{y_k})= \widehat{f}^{i_k+N_k}(\widehat{z_k})$. We have the following result.

\begin{coro}
  For every integer $k$ large enough the paths $I_{\widehat{\mathcal{F}}}^{i_k}(\widehat{f}^{-i_k}(\widehat{y_k}))$ and $I_{\widehat{\mathcal{F}}}^{kN-N_k}(\widehat{f}^{N_k}(\widehat{y_k}))-(k,0)$ intersect $\widehat{\mathcal{F}}$-transversally.
\end{coro}

By density of the set $X_\alpha$ given by Proposition \ref{consequences} and Lemma \ref{lemma10LCT}, considering a point close to $\widehat{y_k}$, we can suppose that $\widehat{y_k}$ belongs to $X_\alpha$. We have the consequence following, which results of Proposition \ref{consequences}.

\begin{lemma}
  For every positive real number $L$, there exist two integers $p\in -\N$ and $q\in\N$ satisfying $q-i_k>L$ such that:
  \begin{itemize}
    \item $\widehat{\gamma_k}|_{[a,a''_k]}$ is a subpath of $I_{\widehat{\mathcal{F}}}^{i_k+2}(\widehat{f}^{q-i_k-1}(\widehat{y_k}))-(p,0)$.
  \end{itemize}
\end{lemma}
\begin{proof}
 By Proposition \ref{consequences}, one can find an integer $l$ sufficiently large such that $q_l>L+i_k$ and such that $\widehat{f}^{q_l}(\widehat{y_k})-(p_l,0)$ is close to $\widehat{y_k}$. Lemma \ref{lemma10LCT} permits us to conclude that the path $I_{\widehat{\mathcal{F}}}^{i_k}(\widehat{f}^{-i_k}(\widehat{y_k}))$ is a subpath of $I_{\widehat{\mathcal{F}}}^{i_k+2}(\widehat{f}^{q_l-i_k-1}(\widehat{y_k}))-(p_l,0)$. Hence choosing $q=q_l$ and $p=p_l$ we have that $\widehat{\gamma_k}|_{[a,a''_k]}$ is a subpath of $I_{\widehat{\mathcal{F}}}^{i_k+2}(\widehat{f}^{q-i_k-1}(\widehat{y_k}))-(p,0)$. This completes the proof of the lemma.
\end{proof}

Let $\widehat{\gamma^{**}}$ be the transverse path $I_{\widehat{\mathcal{F}}}^{q+1-N_k}(\widehat{f}^{N_k}(\widehat{y_k}))$. We deduce the next result.

\begin{coro}
   The paths $\widehat{\gamma^{**}}$ and $\widehat{\gamma^{**}}+(p-k,0)$ intersect $\widehat{\mathcal{F}}$-transversally at some leaf
   $\phi_{\widehat{\gamma^{**}}(t)}=\phi_{\left(\widehat{\gamma^{**}}+(p-k,0)\right)(s)}$ where $s<t$.
\end{coro}
\begin{proof}
 One knows that the path $\widehat{\gamma_k}|_{[a,a''_k]}+(p,0)$ is a subpath of  $I_{\widehat{\mathcal{F}}}^{i_k+2}(\widehat{f}^{q-1-i_k}(\widehat{y_k}))=\widehat{\gamma^{**}}|_{[q-1-i_k-N_k,q+1-N_k]}$ and $\widehat{\gamma_k}|_{[b''_k,b]}+(p-k,0)$ is a subpath of $\widehat{\gamma^{**}}|_{[0, kN-N_k]}+(p-k,0)$, which implies that $\widehat{\gamma^{**}}|_{[q-1-i_k-N_k,q+1-N_k]}$ and $\widehat{\gamma^{**}}|_{[0, kN-N_k]}+(p-k,0)$ have a $\widehat{\mathcal{F}}$-transverse intersection. The result follows since we took $q$ sufficiently large so that $q-1-i_k-N_k$ is larger than $kN-N_k$.
\end{proof}

As a consequence of Proposition \ref{pr:newperiodicforcing2} we deduce the following corollary.

\begin{coro}
  We have that $\frac{p-k}{q+1-N_k}$ belongs to $\rot(\widehat{f})$.
\end{coro}
As $N_k\ge 1$, this implies also that $\frac{p-k}{q+1-N_k}<\frac{p-k}{q}$, but since $p-\alpha q <1$, we have that $\frac{p-k}{q} -\alpha<0$, a contradiction since we assume that $\rot(\widehat{f})=[\alpha,\beta]$.

\section{Proof of Theorem B}\label{sec:proofoftheoremB}

In this section, we will prove Theorem B. We start by proving a result of uniformly boundedness for the diameter of the projection onto the first coordinate of the leaves of $\widehat{\mathcal{F}}$, where $\widehat{\mathcal F}$ is the lift of a foliation that is transverse to a maximal identity isotopy whose endpoint is a homeomorphism that satisfies hypotheses of Theorem B. This result play a key role in the proof of Theorem B.

\subsection{A result of uniformly boundedness}\label{subsection dynamics ... }

 Let $f$ be a homeomorphism of the closed annulus $\Ac:=\T{1}\times [0,1]$ which is isotopic to the identity, that is, $f$ preserves the orientation and each boundary component of $\Ac$. Suppose that $\A:=\T{1}\times (0,1)$ is a Birkhoff region of instability of $f$.  Let $I'$ be an identity isotopy of $f$ and let $\widehat{f}$ be the lift of $f$ to $\R\times [0,1]$ associated to $I'$. Suppose that $\rot(\widehat{f})=[\alpha,\beta]$ with $\alpha<0<\beta$ and that both boundary component rotation numbers are positive.\\

 We consider the open annulus $\A:=\T{1}\times (0,1)$. We will denote by $N$ (resp. by $S$) the upper (resp. lower) end of $\A$. We recall that the homeomorphism $f$ restrict to the open annulus $\A$ can be extended to a homeomorphism, which we still denote $f$, of the end compactification of $\A$, which is a topological sphere, and this homeomorphism fixes both ends of $\A$. Let $I$ be a maximal identity isotopy larger than $I'$ (isotopy associated to the lift $\widehat{f}$) and let $\widehat{I}$ be a lift of $I$. Let $\mathcal{F}$ be a singular foliation transverse to $I$, and let $\widehat{\mathcal{F}}$ be a lift of $\mathcal{F}\vert_{\A}$. We note that $N$ and $S$ are in $\fix(I)$ and that these are isolated singularities of $\mathcal{F}$. The next result follows from Proposition \ref{propofoliationnearsingularity} which describes the dynamics of a foliation near to an isolated singularity and the fact that the boundary component rotation numbers are positive.

\begin{lemma}\label{lemmaSsinkNsource}
 The isolated singularity $S$ (resp. $N$) is a sink (resp. a source) of $\mathcal{F}$.
\end{lemma}
\begin{proof}
  We will prove that $S$ is a sink of $\mathcal{F}$ (one proves analogously that $N$ is a source). By Proposition \ref{propofoliationnearsingularity} it is sufficient to prove that both Cases (1) and (2) of Proposition \ref{propofoliationnearsingularity} do not hold and that $S$ is not a source. We will prove it by contradiction.\\
  Suppose first that Case (1) of Proposition \ref{propofoliationnearsingularity} holds, that is, there exists an open topological disk $D$ containing $S$ and contained in a small neighborhood of $S$ whose boundary is a closed leaf of $\mathcal{F}$. By transversality of the foliation either $f(D)\subset D$ or $f^{-1}(D)\subset D$. This contradicts the fact that $\A$ is a Birkhoff region of instability of $f$.\\
  Suppose now that Case (2) of Proposition \ref{propofoliationnearsingularity} holds, that is, there exist leaves $\phi_S^+$ and $\phi_S^-$ of $\mathcal{F}$ whose $\omega$-limit and $\alpha$-limit set are reduced to $S$ respectively. We will prove that the existence of $\phi_S^-$ implies that the rotation number of the boundary component $\R\times \{0\}$ is negative or zero. This contradicts our hypothesis.

  \begin{claim}
    Suppose that there exists a leaf $\phi_S^-$ of $\mathcal{F}$ whose $\alpha$-limit set is reduced to $S$. Then the rotation number of the boundary component $\R\times \{0\}$ is negative or zero.
  \end{claim}
  \begin{proof}
Let us parameterize the leaf  $\phi_S^-:\R\to\A$. Conjugating $f$ by a homeomorphism given by Schoenflies' Theorem, we may suppose that $\phi_S^-\vert_{(-\infty,0]}$ is contained in $\{0\}\times (0,1)$. Let $U$ be a Euclidean circle centered at $S$ whose boundary meets $\phi^-_S\vert_{(-\infty,0]}$. By Proposition \ref{propolocallytransverse}, $\mathcal{F}$ is locally transverse to $I$ at $S$. Let $V$ be a neighborhood of $S$ contained in $U$ given by the local transversality of $\mathcal{F}$ to $I$ at $S$: the trajectory of each $z\in V$, $z\neq S$ along $I$ is homotopic, with fixed endpoints, to an arc $I^1_{\mathcal{F}}(z)$ which is transverse to $\mathcal{F}$ and included in $U$. In particular, the arc $I^1_{\mathcal{F}}(z)$ must cross $\phi_S^-$ from right to left. More precisely, let $\widehat{f}$ be the lift of $f\vert_\A$ associated to $I$, and let $\widehat{\mathcal{F}}$ be the lift of $\mathcal{F}\vert_\A$. Let $\widehat{U}$ and $\widehat{V}$ be lifts of the sets $U\setminus\{S\}$ and $V\setminus\{S\}$ respectively. Let $\hat{\phi}_S^-$ be the lift of $\phi_S^-$ contained in the line $\{0\}\times (0,1)$. Let $z\in V\setminus \{S\}$ and let $n\in\N$ such that $\{z,\cdots ,f^{n-1}(z)\}\subset V$. Let $\widehat{z}$ be the lift of $z$ such that $-1<p_1(\hat{z})\leq 0$ and $I^1_{\widehat{\mathcal{F}}}(\widehat{z})$ the lift of the arc $I^1_{\mathcal{F}}(z)$ from $\widehat{z}$. Since the path $ I_{\widehat{\mathcal{F}}}^n(\hat{z})$ is transverse to $\widehat{\mathcal{F}}$ and does not meet the boundary of $\widehat{U}$, we obtain that $p_1(\widehat{f}^n(\widehat{z}))<0$ and thus $$\rho_n(\widehat{f},z):=\frac{1}{n}(p_1(\widehat{f}^n(\widehat{z}))-p_1(\widehat{z}))\leq \frac{1}{n}.$$   This implies that the rotation number of the boundary component $\R\times \{0\}$ is negative or zero.
\end{proof}

We note finally that if $S$ is a source, then by the above claim we deduce that the rotation number of the boundary component $\R\times \{0\}$ is negative or zero. This contradicts again our hypothesis. This completes the proof of the lemma.
\end{proof}

Let $\Gamma_S$ and $\Gamma_N$ be two $\mathcal{F}$-transverse loops close enough to $S$ and $N$ respectively, and let $\gamma_S$ and $\gamma_N$ be their respective natural lifts such that the annuli
$$  U_S:= \bigcup_{t\in\R} \phi_{\gamma_S(t)} \quad \text{ and } \quad U_N:= \bigcup_{t\in\R} \phi_{\gamma_N(t)}.$$
coincide with the attracting and repelling basin of $S$ and $N$ for $\mathcal{F}$ respectively. Let $\widehat{\mathcal{F}}$ be a lift of $\mathcal{F}\vert_\A$ to $\R\times (0,1)$. Now we can state the main result of this subsection, which will also be useful in the proof of Theorem B.

\begin{prop}\label{the diamter of the leaves is bounded}
Up to a suitable change of coordinate, the diameter (on the first coordinate) of the leaves of $\widehat{\mathcal{F}}$ are uniformly bounded.
\end{prop}
\begin{proof}
Up to a suitable change of coordinate we can suppose that the foliation $\mathcal{F}$ restrict to a neighborhood of $S$ (resp. $N$) coincides with the foliation of vertical lines downward (resp. upward) on $\T{1}\times (0,1)$.
By Lemma \ref{lm: lemma29} and Corollary \ref{Co: corollary210} (applied at both ends of $\A$), there exist a neighborhood $V_S\subset U_S$ (resp. $V_N\subset U_N$) of $S$ (resp. of $N$) and integers $n_S\geq 1$ and $n_N\geq 1$ such that for every $z\in V_S$, $z\neq S$ (resp. $z\in V_N$, $z\neq N$) the closed path $\gamma_S|_{[0,1]}$ (resp. $\gamma_N|_{[0,1]}$) is a subpath of $I_{\mathcal{F}}^{n_S}(z)$ (resp. $I_{\mathcal{F}}^{n_N}(z)$). On the other hand, since $\A$ is a Birkhoff region of instability of $f$ one can find two points $z_0$ and $z_1$ and integers $n_0\geq 1$ and $n_1\geq 1$ satisfying:
$$     z_0, f^{n_1}(z_1) \in \cap_{i=0}^{n_S} f^{-i}(V_S) \quad  \left(\text{resp.} \,\,   z_1, f^{n_0}(z_0) \in \cap_{i=0}^{n_N} f^{-i}(V_N) \right) $$
We will write $\gamma_0:=I_{\mathcal{F}}^{n_0+n_S}(z_0)$ and $\gamma_1:=I_{\mathcal{F}}^{n_1+n_N}(z_1)$ for convenience.
Therefore there exist leaves $\phi_S\subset  U_S$ and $\phi_N\subset  U_N$ of $\mathcal{F}$ and real numbers $s_0<t_0$ and $s_1<t_1$ such that:
\begin{itemize}
  \item $\phi_{\gamma_0(s_0)}=\phi_S$ and $\phi_{\gamma_0(t_0)}=\phi_N$;
  \item $\phi_{\gamma_1(s_1)}=\phi_N$ and $\phi_{\gamma_1(t_1)}=\phi_S$.
\end{itemize}
Replacing $\gamma_0$ by an equivalent transverse path, one can suppose that $\gamma_0(s_0)=\gamma_1(t_1)$ and $\gamma_0(t_0)=\gamma_1(s_1)$. Let $\Gamma$ be the loop naturally defined by the closed path $\gamma_0|_{[s_0,t_0]}\gamma_1|_{[s_1,t_1]}$ which is transverse to $\mathcal{F}$. Since the loop $\Gamma$ is homologous to zero in the sphere, one can define a \textit{dual function} $\delta$ defined up to an additive constant on $\s^2\setminus \Gamma$ as follows: for every $z$ and $z'$ in $\s^2\setminus \Gamma$, the difference $\delta(z)-\delta(z')$ is the algebraic intersection number $\Gamma \wedge \gamma'$, where $\gamma'$ is any path from $z$ to $z'$. As $\Gamma$ is transverse to $\mathcal{F}$ the function $\delta$ decreases along each leaf with a jump at every intersection point. One proves that $\delta$ is bounded and that the space of leaves that meet $\Gamma$, furnished with the quotient topology is a (possibly non Hausdorff) one dimensional manifold (see \cite{lct} for more details). In particular, there exists an integer $K\geq 1$ such that $\Gamma$ intersects each leaf of $\mathcal{F}$ at most $K$ times. Moreover any lift of the set $\phi_{S}\cup \gamma_0|_{[s_0,t_0]}\gamma_1|_{[s_1,t_1]} \cup \phi_N$ separates the plane $\R^2$. More precisely, every set whose diameter on the first coordinate is large enough must intersect a lift of $\phi_{S}\cup \gamma_0|_{[s_0,t_0]}\gamma_1|_{[s_1,t_1]} \cup \phi_N$. Hence, we have that each leaf of $\widehat{\mathcal{F}}$ intersects at most $K$ translates of a lifted of the closed path $\gamma_0|_{[s_0,t_0]}\gamma_1|_{[s_1,t_1]}$. Hence one deduces that for every leaf $\hat{\phi}$ of $\widehat{\mathcal{F}}$ the diameter of the projection on the first coordinate of $\hat{\phi}$ is smaller than the constant $\diam(\gamma_0|_{[s_0,t_0]}\gamma_1|_{[s_1,t_1]})+K+2$. This completes the proof of the proposition.
\end{proof}

\subsection{Proof of Theorem B}

Let $f$ be a homeomorphism of the closed annulus $\Ac:=\T{1}\times [0,1]$ which is isotopic to the identity. Let $\hat{f}$ be a lift of $f$ to $\R\times [0,1]$. We assume that $\rot(\hat{f})=[\alpha, \beta]$ and that the rotation number of both boundary components are strictly larger than $\alpha$, the case where the rotation number of both boundary components are strictly smaller than $\beta$ is similar. Hence considering a rational number $p/q$ between the left endpoint of $\rot(\hat{f})$ and the minimum of the two boundary component rotation numbers, we can replace $f$ by a power $f^q$ and the lift $\hat{f}$ by a lift $\hat{f}^q+(p,0)$, and so suppose that $\rot(\hat{f})=[\alpha,\beta]$ with $\alpha<0<\beta$ and that both boundary component rotation numbers are strictly positive.

We consider the open annulus $\A:=\T{1}\times (0,1)$. We will denote by $N$ (resp. $S$) the upper (resp. lower) end of $\A$. We recall that the homeomorphism $f$ restricted to the open annulus $\A$ can be extended to a homeomorphism, denoted still $f$, of the end compactification of $\A$, which is a topological sphere, and this homeomorphism fixes both ends of $\A$. Let $I$ be a maximal identity isotopy larger than $I'$ (isotopy associated to the lift $\hat{f}$) and let $\hat{I}$ be a lift of $I$. Let $\mathcal{F}$ be a singular foliation transverse to $I$, and let $\widehat{\mathcal{F}}$ be a lift of $\mathcal{F}\vert_\A$.\\

We know from the previous subsection that $S$ and $N$ are a sink and a source of $\mathcal{F}$ respectively and that up to a  conjugation the leaves of $\widehat{\mathcal{F}}$ are uniformly bounded on the first coordinate. Moreover for a positive real number $\delta$ the foliation $\widehat{\mathcal{F}}$ restrict to $\R\times (0,\delta)$ (resp. $\R\times (1-\delta,1)$) is the foliation in vertical lines on $\R\times (0,1)$ oriented downwards (resp. upwards).  Let $\fonc{\hat{\gamma}_N}{\R}{\R\times (0,1)}$ and $\fonc{\hat{\gamma}_S}{\R}{\R\times (0,1)}$ be the transverse path defined by
$$ \hat{\gamma}_N(t):= (t,1-\delta/2) \quad \text{ and } \quad \hat{\gamma}_S(t):= (t,\delta/2).$$
Let us consider
$$\widehat{U}_N:=\cup_{t\in\R} \phi_{\hat{\gamma}_N(t)} \quad \text{ and } \quad \widehat{U}_S:=\cup_{t\in\R} \phi_{\hat{\gamma}_S(t)}.$$

We will begin by proving the following result.

\begin{lemma}\label{lemma1prooftheoB}
  There exist two admissible transverse paths $\hat{\gamma}^*_0:[a_0,b_0]\to \widehat{\A}$ and $\hat{\gamma}^*_1:[a_1,b_1]\to \widehat{\A}$, real numbers $a_0<t_0<b_0, \, a_1<t_1<b_1$ and a real number $K^*>0$ such that:
  \begin{itemize}
    \item[(i)] $\hat{\gamma}^*_0\vert_{[a_0,t_0]}$ and $\hat{\gamma}^*_1\vert_{[t_1,b_1]}$ intersect $\hat{\mathcal{F}}$-transversally; and
    \item[(ii)] for every transverse path $\hat{\gamma}:[a,b]\to \widehat{\A}$ such that $p_1(\hat\gamma(b)-\hat\gamma(a))<-K^{*}$, there exist two integers $p_0$ and $p_1$ such that $\hat{\gamma}$ intersects $\hat{\mathcal{F}}$-transversally both $\left(\hat{\gamma}^*_0+(p_0,0)\right)\vert_{[t_0, b_0]}$ and $\left(\hat{\gamma}^*_1+(p_1,0)\right)\vert_{[a_1,t_1]}$.
  \end{itemize}
\end{lemma}
\begin{proof}
\textit{Let us prove (i).}
As in the proof of Proposition \ref{the diamter of the leaves is bounded}, one can find two points $\hat{z}_0$, $\hat{z}_1$ in $\widehat{\A}=\R\times (0,1)$ and two positive integers $n_0$ and $n_1$ such that, if $\hat{\gamma}^*_0:=I_{\hat{\mathcal{F}}}^{n_0}(\hat{z}_0)$ and $\hat{\gamma}^*_1:=I_{\hat{\mathcal{F}}}^{n_1}(\hat{z}_1)$, then there exist real numbers
$$a_0< s^{-}_0<s^{+}_0<t_0\le r_0^{-}<r_0^{+}<b_0$$ satisfying:
\begin{itemize}
  \item $\hat{\gamma}^*_0|_{[s_0^{-},s_0^{+}]}$ is equivalent to the path $\hat{\gamma}_S|_{[-1,0]}$;
  \item $\hat{\gamma}^*_0|_{(s_0^{+},t_0)}$ is contained in $\widehat{U}_{S}$ but it do not meet $\phi_{\widehat{\gamma}_S(1)}$;
  \item $\hat{\gamma}^*_0(t_0)$ does not belong to $\widehat{U}_{S}$;
  \item $\hat{\gamma}^*_0|_{[r_0^{-},r_0^{+}]}$ belongs to the complement of $\widehat{U}_{S}\cup\widehat{U}_{N}$;
  \item If $\varepsilon$ is sufficiently small, then $\hat{\gamma}^*_0(r_0^{-}-\varepsilon)\in \widehat{U}_{S}$ and $\hat{\gamma}^*_0(r_0^{+}+\varepsilon)\in \widehat{U}_{N}$.
\end{itemize}
and real numbers
$$a_1< r_1^{-}<r_1^{+}<t_1\le s_1^{-}<s_1^{+}<b_1$$
satisfying:
\begin{itemize}
  \item $\hat{\gamma}^*_1|_{[s_1^{-},s_1^{+}]}$ is equivalent to the path $\hat{\gamma}_S|_{[0,1]}$;
  \item $\hat{\gamma}^*_1|_{(t_1,s_1^{-})}$ is contained in $\hat{U}_{S}$ but it do not meet $\phi_{\hat{\gamma}_S(-1)}$;
  \item $\hat{\gamma}^*_1(t_1)$ does not belong to $\widehat{U}_{S}$;
   \item $\hat{\gamma}^*_1|_{[r_1^{-},r_1^{+}]}$ belongs to the complement of $\widehat{U}_{S}\cup\widehat{U}_{N}$;
  \item If $\varepsilon$ is sufficiently small, then $\hat{\gamma}^*_1(r_1^{-}-\varepsilon)\in \widehat{U}_{N}$ and $\hat{\gamma}^*_1(r_1^{+}+\varepsilon)\in \widehat{U}_{S}$.
\end{itemize}
We can also assume, by using Corollary \ref{corollary24LCT}, that neither the path $\hat{\gamma}^*_0\vert_{[a_0, b_0]}$ nor the path $\hat{\gamma}^*_1\vert_{[a_1,b_1]}$ have $\widehat{\mathcal{F}}$-transverse self-intersections.
We note that the paths $\hat{\gamma}^*_0\vert_{[s_0^{-},t_0]}$ and $\hat{\gamma}^*_1\vert_{[t_1,s_1^{+}]}$ have a $\hat{\mathcal{F}}$-transverse intersection at $\phi_{\hat{\gamma}_S(0)}$, and so item (i) is proved.\\

\textit{Let us prove (ii).} Let $\varepsilon$ be sufficiently small such that $\hat{\gamma}^*_0(t)$ belongs to $\widehat{U}_{S}$ for any $r_0^{-}-\varepsilon<t<r_0^{-}$ and such that $\hat{\gamma}^*_0(t)$ belongs to $\widehat{U}_{N}$ for any $r_0^{+}<t<r_0^{+}+\varepsilon$.  By Proposition \ref{the diamter of the leaves is bounded} there exists a real number $K_0>0$ such that for each leaf $\hat{\phi}$  of $\widehat{\mathcal{F}}$ the diameter of $p_1(\hat{\phi})$ is bounded by $K_0$ and by compactness there exists a real number $K'_0>0$ such that $\hat{\gamma}^*_0\vert_{[r_0^{-}-\varepsilon,r_0^{+}+\varepsilon]}$ is contained in $(-K'_0,K'_0)\times (0,1)$. Let $K=K_0+K'_0$, and consider a transverse path $\hat{\gamma}$ with diameter on the first coordinate larger than $K^*:=2K+1$. Then there exists an integer $p_0$ such that $\hat{\gamma}$ meets both $(-\infty, p_0-K)\times (0,1)$ and $(p_0+K,+\infty)\times (0,1)$.

Let $B$ be the union of leafs met by $\hat{\gamma}^*_0\vert_{(r_0^{-}-\varepsilon,r_0^{+}+\varepsilon)}$. Note that $B$ is a foliated subset of $\widehat{\A}$ homeomorphic to the plane, and as $\hat{\gamma}^*_0\vert_{[r_0^{-}-\varepsilon,r_0^{+}+\varepsilon]}$ has no $\widehat{\mathcal{F}}$-transverse self-intersection, the space of leafs of $\widehat{\mathcal{F}}$ in $B$ is homeomorphic to an open interval. Furthermore,  since $B$ contains a leaf of $\widehat{U}_{S}$ and a leaf of $\widehat{U}_{N}$,  $B$ separates $\R^{2}$ and its complement has exactly two connected components, one denoted $L(B)$ that contains $(-\infty, -K)\times (0,1)$ and the other, denoted $R(B)$, contains $(K,+\infty)\times (0,1)$. Finally, note that $\phi_{\gamma^*_0(r_0^{-}-\varepsilon)}$ belongs to $\partial L(B)$ and that $B$ is locally to the left of this leaf. Likewise,  $\phi_{\gamma^*_0(r_0^{+}+\varepsilon)}$ belongs to $\partial R(B)$ and that $B$ is locally to the right of this leaf (see Figure \ref{fig:the set B proof of Theorem B}).

Since $\hat{\gamma}(a)$ belongs to $R(B)-(p_0,0)$ and $\hat \gamma(b)$ belongs to $L(B)-(p_0,0)$, one find some $a<a'<b'<b$ such that $\hat{\gamma}(a')$ belongs to $\partial\left(R(B)-(p_0,0)\right)$, $\hat{\gamma}(b')$ belongs to $\partial\left(L(B)-(p_0,0)\right)$, and such that $\hat\gamma\vert_{(a',b')}$ is contained in $B-(p_0,0)$. Note that $\phi_{\hat\gamma(a')}+(p_0,0)$ is not $\phi_{\hat{\gamma}^*_0(r_0^{+}+\varepsilon)}$, since the latter has $B$ on its right, and likewise $\phi_{\hat\gamma(b')}+(p_0,0)$ is not $\phi_{\hat{\gamma}^*_0(r_0^{-}-\varepsilon)}$, since the latter has $B$ on its left. One concludes that $\widehat{\gamma}\vert_{[a',b']}$ has a $\widehat{\mathcal{F}}$-transverse intersection with  $\left(\hat{\gamma}^*_0\vert_{[r_0^{-}-\varepsilon,r_0^{+}+\varepsilon]}+(p_0,0)\right)$
  at any leaf $\phi_{\hat\gamma(t)}$ for $a'<t<b'$. Since the same construction holds for every $0<\varepsilon'<\varepsilon$, we get that $\hat\gamma$ and $\hat{\gamma}^*_0\vert_{[r_0^{-},r_0^{+}]}+(p_0,0)$ intersect $\widehat{\mathcal{F}}$-transversally, as claimed. The $\widehat{\mathcal{F}}$-transverse intersection with $\hat{\gamma}^*_1\vert_{[r_1^{-},r_1^{+}]}+(p_1,0)$ can be obtained in a similar way.
\end{proof}

\begin{center}
\begin{figure}[h!]
  \centering
   \includegraphics{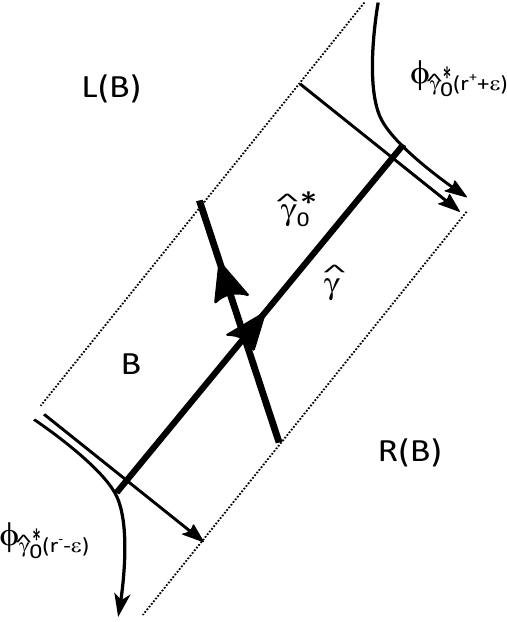}
  \caption{The set $B$.}
  \label{fig:the set B proof of Theorem B}
\end{figure}
\end{center}

\begin{proof}[End of the proof of Theorem B]
One knows by Proposition \ref{the diamter of the leaves is bounded} that the leaves of $\widehat{\mathcal{F}}$ are uniformly bounded on the first coordinate. Hence in order to prove Theorem B it is sufficient to prove that there exists a real constant $L>0$ such that for every admissible transverse path $\hat{\gamma}: [a,b]\to \widehat{\A}$ of order $n\geq 1$, one has
$$  p_1(\hat{\gamma}(b))-p_1(\hat{\gamma}(a))-n\alpha \geq -L.$$

Let $\fonc{\hat{\gamma}}{[a,b]}{\R\times (0,1)}$ be a transverse path such that
$$  p_1(\hat{\gamma}(b))- p_1(\hat{\gamma}(a)) <-2K^*  .$$
One can find $c,d$ in $(a,b)$ with $c<d$ such that
$$  p_1(\hat{\gamma}(b))- p_1(\hat{\gamma}(d)) =-K^* \quad \text{ and } \quad  p_1(\hat{\gamma}(c))- p_1(\hat{\gamma}(a)) =-K^*  .$$
By item (ii) from Lemma  \ref{lemma1prooftheoB} there exist $p_0$ and $p_1$ in $\Z$, $a<l_0<c$, $d<l_1<b$, $r_0^{-}<w_0<r_0^{+}$, $r_1^{-}<w_1<r_1^{+}$  such that
\begin{itemize}
  \item $\hat{\gamma}|_{[a,c]}$ and $\hat{\gamma}^*_0+(p_0,0)$ intersect $\widehat{\mathcal{F}}$-transversally at $\hat{\gamma}(l_0)=\hat{\gamma}^*_0(w_0)+(p_0,0)$;
  \item $\hat{\gamma}|_{[d,b]}$ and $\hat{\gamma}^*_1+(p_1,0)$ intersect $\widehat{\mathcal{F}}$-transversally at $\hat{\gamma}(l_1)=\hat{\gamma}^*_1(w_1)+(p_1,0)$.
\end{itemize}
If $\hat{\gamma}$ is admissible of order $n\geq 1$, then the path
 $$  \hat{\gamma}':= \left(\hat{\gamma}^*_0|_{[a_0,w_0]}+(p_0,0) \right) \hat{\gamma}|_{[l_0,l_1]} \left(\hat{\gamma}^*_1|_{[w_1,b_1]}+(p_1,0) \right)$$
 is admissible of order $n+n_0+n_1$ by Corollary \ref{corollary22LCT} and one has
$$  p_1(\hat{\gamma}(b))- p_1(\hat{\gamma}'(1)) <-K^* \quad \text{ and } \quad  p_1(\hat{\gamma}'(0))- p_1(\hat{\gamma}(a)) <-K^*$$

 Recall that $\hat{\gamma}^*_0$ and $\hat{\gamma}^*_1$ intersect $\widehat{\mathcal{F}}$-transversally (item (i) from  Lemma \ref{lemma1prooftheoB}). One deduces that $\hat{\gamma}'$ intersects $\widehat{\mathcal{F}}$-transversally $\hat{\gamma}'+(p',0)$, where $p'=p_0-p_1$. Proposition \ref{pr:newperiodicforcing2} tells us that $p'/ (n+n_0+n_1)$ belongs to $\rot(\hat{f})$, which implies that $$\alpha \leq \frac{p'}{n+n_0+n_1}.$$
 We write $K^{**}$ by the diameter on the first coordinate of $\hat{\gamma}_0^*$. Observe now that
 $$  p_1(\hat{\gamma}'(1))+p'- p_1(\hat{\gamma}'(0)) >-K^{**}.$$
So, one deduces
$$  p_1(\hat{\gamma}(b))-p_1(\hat{\gamma}(a))-n\alpha \geq -2K^* - K^{**} +\alpha n_0+\alpha n_1.$$
This completes the proof of Theorem B.
\end{proof}

\section{Example}\label{sec:example}

Given a homeomorphism of the closed annulus $\Ac:=\T{1}\times [0,1]$ that is isotopic to the identity and having $\A:=\T{1}\times (0,1)$ as Birkhoff region of instability, and a lift $\hat{f}$ of $f$ to $\R\times [0,1]$ with a nontrivial rotation set, $\rot(\hat{f})=[\alpha,\beta]$, and such that the rotation numbers of both boundary components of $\Ac$ are in the interior of $[\alpha,\beta]$. It follows of Theorem B that there exists a real constant $L>0$ such for every $\hat{z}\in\R\times [0,1]$ and every integer $n\geq 1$ we have
  $$ p_1(\hat{f}^{n}(\hat{z}))-p_1(\hat{z})-\alpha n \geq -L \quad \text{ and } \quad p_1(\hat{f}^{n}(\hat{z}))-p_1(\hat{z})-\beta n \leq L.$$

The following example shows that the hypothesis ``the rotation numbers of both boundary components of $\Ac$ are contained in the interior of the rotation set'' is essential in the conclusion of Theorem B.

\begin{prop}
  There exists a homeomorphism $f$ of the closed annulus $\Ac$ which is isotopic to the identity, such that $\Ac$ is a Birkhoff region of instability of $f$ and that has a lift $\hat{f}$ to $\R\times [0,1]$ satisfying:
  \begin{itemize}
    \item[(i)] $\rot(\hat{f})=[0,1]$, and
    \item[(ii)] for every real number $L>0$ there exist a point $\hat{z}$ in $\R\times [0,1]$ and an integer $n$ such that
  $$ p_1(\hat{f}^{n}(\hat{z}))-p_1(\hat{z}) < -L$$
  \end{itemize}
\end{prop}
\begin{proof}
For every real number $r$, we write $T_r$ the homeomorphism of $\R\times [0,1]$ defined by $T_r : (x,y)\mapsto (x+r,y)$. Let $h:(0,1)\to \R$ be such that $h(y)=1/y$ if $y<1/2$ and $h(y)=2$ if $y\ge 1/2$.

For every $ y\in(0,1)$, let
$$I'_y:=\left(h(y)+\frac{1}{16},h(y)+\frac{3}{16}\right)\times \{y\}  \subset \left(h(y),h(y)+\frac{1}{4}\right)\times \{y\}:=I_y$$
be two subsets of $\R\times [0,1]$ and consider
$$\hat{U}'_0:= \bigcup_{y\in (0,1)} I'_y  \subset \bigcup_{y\in (0,1)} I_y:=\hat{U}_0 \quad \text{ and } \quad
\hat{V}'_0 := T_{1/2}(\hat{U}'_0) \subset T_{1/2}(\hat{U}_0):= \hat{V}_0. $$

Let finally,
$$\hat{U}':= \bigcup_{n\in \Z} T_1^n(\hat{U}'_0)  \subset \bigcup_{n\in \Z} T_1^n(\hat{U}_0):=\hat{U} \quad \text{ and } \quad
\hat{V}' := \bigcup_{n\in \Z} T_1^n(\hat{V}'_0)  \subset \bigcup_{n\in \Z} T_1^n(\hat{V}_0):=\hat{V}. $$

Let $\fonc{g}{[0,1]}{[0,1]}$ be a homeomorphism of $[0,1]$ satisfying:
\begin{itemize}
  \item  $g(0)=0$, $g(1)=1$;
  \item for every $y\in (0,1)$ we have $g(y)<y$;
  \item { $\lim\limits_{y\to 0} \left(\frac{1}{y}- \frac{1}{g(y)}\right)=0$.}
\end{itemize}
We note that the dynamics of $g$ is well-known, that is, for every $y\in(0,1)$ we have
\begin{equation}\label{equ dynamics of g}
 \lim_{n\to +\infty}  g^n(y)=0 \quad \text{ and } \quad  \lim_{n\to +\infty}  g^{-n}(y)=1.
\end{equation}

To construct $\hat{f}$ we start with a homeomorphism $\hat{f}_0$ satisfying:
$$ \hat{f}_0(x,y)=\begin{cases}
  \left(x-h(y)+h(g(y)),g(y)\right),& \text{ if} \quad (x,y) \in \hat{U}'_0\subset \hat{U}_0,\\
  \left(x-h(y)+h(g^{-1}(y)),g^{-1}(y)\right),& \text{ if} \quad (x,y) \in \hat{V}'_0\subset \hat{V}_0.\\
\end{cases}$$
 Note that $\hat{f}_0$ leaves both $\hat{U}'_0$ and $\hat{V}'_0$ invariant. Let us begin by extending $\hat{f}_0$ to $\hat{U}_0$ and to $\hat{V}_0$ such that $\hat{f}_0$ leaves both this sets invariant, such that  $$ \begin{cases}
  p_2(\hat{f}_0(x,y))<y,& \text{if} \quad (x,y) \in \hat{U}_0,\\
  y< p_2(\hat{f}_0(x,y))  ,& \text{if}\quad (x,y) \in \hat{V}_0,\\
\end{cases}$$
where $p_2:\R\times [0,1] \to [0,1]$ is the projection on the second coordinate, such that $\hat{f}_0$ extends continuously as the identity to the boundary of $\hat{U}_0\cup \hat{V}_0\cap \A$ and finally that $p_1(\hat{f}_0(x,y))=x$ if $g(y)\ge 1/2$. Now extend $\hat{f}_0$ to $\hat{U}$ and $\hat{V}$ by the formula $\hat{f}_0 T_1=T_1 \hat{f}_0$. Finally we extend $\hat{f}_0$ to $\R\times [0,1]$ such that $\hat{f}_0$ coincides with the identity on the complement of the union of $\hat{U}$ and of $\hat{V}$. Note that $\hat{f}_0$ extends continuously to $\R\times\{0,1\}$ since  $\lim_{y\to 0} (h(y)-h(g(y)))=\lim_{y\to 1} (h(y)-h(g(y)))=0$.

We note that by construction $\hat{f}_0$ commutes with $T_1$, and that for every $i\in \Z$ the sets
$T_i(\hat{U}'_0)$ and $T_i(\hat{V}'_0)$ are $\hat{f}_0$-invariant. Hence it is easy to check that if $(x,y)\in \hat{V}'$ and $n\geq 1$, then
\begin{equation*}
\hat{f}_0^n(x,y)= \left(x-\frac{1}{y}+\frac{1}{g^{-n}(y)},g^{-n}(y)\right).
\end{equation*}
Thus,
$$ p_1(\hat{f}_0^n(x,y))-x = \frac{1}{g^{-n}(y)}-\frac{1}{y},$$
so one knows that the function $p_1(\hat{f}_0^n(x,y)- (x,y))$ is not bounded from below.
Moreover from \eqref{equ dynamics of g}, $\A=\T{1}\times (0,1)$ is a Birkhoff region of instability of $f_0$, the projection of $\hat{f}_0$ to $\Ac$. More precisely, the $f_0$-orbit of each point $z\in U'=\hat{\pi}(\hat{U}'_0)$ goes from the upper boundary component of $\Ac$ to the lower one, that is
  \begin{equation}\label{equ A region insta f0}
        \lim_{n\to +\infty } p_2(f_0^{-n}(z))=1 \quad \text{ and } \quad \lim_{n\to +\infty } p_2(f_0^{n}(z))=0.
  \end{equation}
and the $f_0$-orbit of each point $z\in V'=\hat{\pi}(\hat{V}'_0)$ goes from the lower boundary component of $\Ac$ to the upper one.
Let $y_1=\frac{g^{-1}(1/2)+g^{-2}(1/2)}{2}$. Consider $\hat{f}_1(x,y):=T_{\varphi(y)}(x,y)$, where $\varphi:[0,1]\to [0,1]$ is a continuous function satisfying:
\begin{itemize}
  \item $\varphi(y)=0$ if $y\in [0,g^{-1}(1/2)]\cup [g^{-2}(1/2),1]$; and
  \item $\varphi(y_1)=1$.
\end{itemize}
We note that $\hat{f}_1$ commutes with $T_1$, the compact strip $$S_0:=\R\times[g^{-1}(1/2),g^{-2}(1/2)]$$ is an $\hat{f}_1$-invariant set, and $\hat{f}_1$ acts as the identity on the complement of $S_0$. We now consider $\hat{f}:= \hat{f}_1\circ \hat{f}_0$ which is a homeomorphism of $\R\times [0,1]$ and which commutes with $T_1$. It remains to prove that $\hat{f}$ satisfies the properties described in the proposition.\\

To see that $\A$ is a Birkhoff region of instability of $f$, note that if we choose some point $\hat z_0=(x_0,g^{-1}(1/2))\in \hat{V}'$ then $\hat{f}^{n}(\hat z_0)$ belongs to $\R\times [g^{-2}(1/2), 1]$ for all $n>0$, and to $\R\times [0, g^{-1}(1/2)]$ for all $n\le 0$. In particular,  $\hat{f}^{n}(\hat z_0)=\hat{f}_0^{n}(\hat z_0)$. Therefore $z_0=\hat{\pi}(\hat z_0)$ goes from the lower boundary to the upper boundary component of $\Ac$. We can likewise choose a point $\hat z_1 = (x_1, g^{-1}(1/2))\in \hat{U}'$ and then $\hat{f}^{n}(\hat z_1)$ belongs to $\R\times [g^{-2}(1/2), 1]$ for all $n<0$, and to $\R\times [0, g^{-1}(1/2)]$ for all $n\ge 0$.  Therefore $z_1=\widehat{\pi}(\hat z_1)$ goes from the upper boundary component to the lower boundary component of $\Ac$.

We note also that any point in $\R\times [0, g^{-1}(1/2)]$ belongs to either $\hat V$, to $\hat U$ or is a fixed point of $\hat f_0$. One verifies trivially that the set $\R\times [0, g^{-1}(1/2)]\cap \hat V$ is forward invariant for $\hat f_0$ and therefore also for $\hat f$, and every point is a wandering point. One also verifies that  $\R\times [0, g^{-1}(1/2)]\cap \hat U$ is backward invariant for $\hat f$, and thus all its points are also wandering. Therefore any point in $\R\times [0, g^{-1}(1/2)]$ that is not wandering must be fixed by $\hat f$. A similar argument shows that any point of $\R\times [g^{-2}(1/2), 1]$ that is not wandering for $\hat f$ is also fixed. Thus any periodic point of $f$ that does not lift to a fixed point of $\hat f$ must have its whole orbit contained in the annulus $\T{1}\times[g^{-1}(1/2),g^{-2}(1/2)]$. Now, for any point $(x,y)$ in $S_0$, we have that $0\le p_1(\hat f(x,y)-(x,y))\le 1$, so one gets that $\rot(\hat{f})\subset [0,1]$. Finally, take the point $(1/2,y_1)$, which lies in the boundary of $\hat U$. Therefore $\hat f_0(1/2,y_1)=(1/2,y_1)$ and therefore $\hat{f}(1/2,y_1)=(1/2+1,y_1)$. Therefore $1\in \rot(\hat{f})$ and since the closed annulus is a region of instability, we get that $\rot(\hat{f})$ is a closed interval containing both $0$ and $1$. We deduce that $\rot(\hat{f})= [0,1]$.
\end{proof}

\section{Realization results}\label{sec:realizationresults}

In this section, we will prove Theorems C and D.

\subsection{Proof of Theorem C}

Theorem C will be a consequence of the following stronger proposition, since being a Mather region of instability is stronger than being a $SN$ mixed region of instability.

\begin{prop}\label{pr:realizationSN}
Let $f:\Ac\to\Ac$ be a homeomorphism which is isotopic to the identity, and let $\widehat f$ be a lift of $f$ to the universal covering. Suppose that $\T{1}\times(0,1)$ is a $SN$ mixed region of instability. For every $\rho$ in $\rot(\widehat{f})$ there exists a compact invariant set $Q_\rho$ such that for every point of $Q_\rho$ has a rotation number well-defined and it is equal to $\rho$. Moreover, if $\rho=p/q$ is a rational number, written in an irreducible way, then  $Q_\rho$ is the orbit of a period point of period $q$.
\end{prop}

The proof of the proposition is immediate when $\rot(\widehat f)$ is a single point, as in this case every point of $\Ac$ has a well defined rotation number and this number is unique. Therefore we can assume that $\rot(\widehat f)=[\alpha, \beta]$, with $\alpha<\beta$.
We divide the proof of Proposition \ref{pr:realizationSN} in two cases, when $\rho$ is in the boundary of the rotation set, and when it is in the interior.

\subsubsection{When $\rho$ is a boundary point}

 We show that $\alpha$ is realized by a compact invariant set, the case for $\beta$ is similar. If any of the boundary components rotation numbers is $\alpha$ it suffices to take $Q_\alpha$ as the boundary component, therefore we assume that $\alpha$ is strictly smaller than the rotation numbers of the boundary components and thus that $f$ satisfies the hypotheses of Theorem B, since being a SN mixed region of instability is stronger than being a Birkhoff region of instability. Let us consider
$$ \mathcal{M}_{\alpha}:=\{ \mu \in\mathcal{M}_{f}(\Ac) : \rot(\mu)=\alpha   \} \quad \text{ and } \quad X_\alpha:= \overline{ \bigcup_{\mu \in \mathcal{M}_\alpha} \supp(\mu)}.   $$
By the following proposition we can take $Q_\alpha$ as the set $X_\alpha$, completing the proof of the proposition in this case

\begin{prop}
  Let $\alpha$ be in the boundary of the rotation set of $\widehat f$. Every measure supported on $X_\alpha$ belongs to $\mathcal{M}_\alpha$. Moreover, if $\widehat{z}$ lifts a point $z$ of $X_\alpha$, then for every integer $n\geq 1$, we have $$  \abs{p_1(\widehat{f}^n(\widehat{z}))-p_1(\widehat{z})-n\alpha} \leq L,  $$
  where $L$ is the constant given by Theorem B.
\end{prop}
\begin{proof}
  We note that it is sufficient to prove the second statement. We recall that, if $\mu \in \mathcal{M}_\alpha$, then every ergodic measure $\nu$ that appears on the ergodic decomposition of $\mu$ also belongs to $\mathcal{M}_\alpha$. It is sufficient to prove that for every ergodic measure $\nu$ in $\mathcal{M}_\alpha$, there exists a set $A$ of full measure such that if $\widehat{z}$ lifts a point $z$ of $X_\alpha$, then for every integer $n\geq 1$, we have
$$  \abs{p_1(\widehat{f}^n(\widehat{z}))-p_1(\widehat{z})-n\alpha} \leq L,$$
  since this implies that the above inequality must hold in the whole support of $\nu$.

  By Atkinson's Lemma, Proposition \ref{lemmaatkinson}, there exists a set $A$ of $\nu$-full measure, such that if $\widehat{z}$ lifts a point $z$ of $A$, then there exists a subsequence of integer $(q_l)_{l\in\N}$ such that
    $$  \lim_{l\to +\infty} \widehat{f}^{q_l}(\widehat{z})- \widehat{z}-(q_l\alpha,0)=0.     $$
  On the other hand, by Theorem B we know that for every $\widehat{z}\in \R\times [0,1]$, and every integer $n\geq 1$ we have
    \begin{equation}\label{equ3}
    p_1(\widehat{f}^n(\widehat{z}))-p_1(\widehat{z})-n\alpha \geq - L.
    \end{equation}
     It remains to check that, if $\widehat{z}$ is a lift of $z\in A$, then
       \begin{equation}\label{equ4}
 p_1(\widehat{f}^n(\widehat{z}))-p_1(\widehat{z})-n\alpha \leq  L.
    \end{equation}
     Indeed, if $l$ is large enough such that $q_l$ is greater than $n$, one can write
  \begin{align*}
     & p_1(\widehat{f}^n(\widehat{z}))-p_1(\widehat{z})-n\alpha=\\
     & = p_1(\widehat{f}^{q_l}(\widehat{z}))-p_1(\widehat{z})-q_l\alpha - \left( p_1(\widehat{f}^{q_l-n}(\widehat{f}^n(\widehat{z})))-p_1(\widehat{f}^n(\widehat{z}))-(q_l-n)\alpha    \right); \\      & \leq p_1(\widehat{f}^{q_l}(\widehat{z}))-p_1(\widehat{z})-q_l\alpha  +  L.
  \end{align*}
    Letting $l$ tend to $+\infty$, we obtain inequality \eqref{equ4}. The proposition follows of inequalities \eqref{equ3} and \eqref{equ4}.
\end{proof}

\subsubsection{When $\rho$ is an interior point}

We already know the existence of an $f$-invariant compact set with rotation number equal to $\rho$ if $\rho\in\{\alpha, \beta\}$ or if $\rho$ is the rotation number of one of the boundary components of $\Ac$. So we assume $\rho\in (\alpha, \beta)$ and $\rho$ is not the rotation number of the upper boundary. We can also assume that the rotation number of the upper boundary is $\rho_1$ larger than $\rho$, the case where it is smaller is again similar. We pick some rational number $p/q$ such that
$$\alpha<p/q<\rho<(p+1)/q<(p+2)/q<\rho_1\le \beta.$$

Let $g=f^q$, and $\widehat g =\widehat f^q-(p,0)$. Let $I'$ be an identity isotopy of $g$, such that its lift to $\widehat{\Ac}$, $\widehat{I'}$, is an identity isotopy of $\widehat{g}$. By Theorem \ref{existence maximal isotopy} one can find a maximal identity isotopy larger than $I'$. By Theorem \ref{existence transverse foliation} one can find an oriented singular foliation $\mathcal{F}$ on $\A$ which is transverse to $I$, its lift to $\widehat{\Ac}$, denoted by $\widehat{\mathcal{F}}$ is transverse to $\widehat{I}$ (the lift of $I$). Note that the rotation number of the upper boundary for $\widehat g$ is larger than 2. Our goal is to show that there exist some integer $n>0$ and some transverse trajectory $\widehat \gamma$ such that $\widehat \gamma$ is admissible of order $n$ and such that $\widehat \gamma$ has a $\widehat{\mathcal{F}}$-transverse intersection with $\widehat \gamma+ (j,0)$ where $j>n$. If we do this, Proposition \ref{pr:newperiodicforcing2} implies that for all $0<\theta\le j/n$, there exists a $g$-invariant compact set with rotation number $\theta$ for $\widehat g$, and we deduce that $f$ has a compact set with rotation number $\rho$.\\
To construct the transverse path $\gamma$, let us first note that: the transverse trajectory of a point $z$ in the upper boundary of the annulus is equivalent to the canonical lift $\gamma'$ of a simple transverse loop $\Gamma'$, and if $\widehat{\gamma'}$ is the lift of $\gamma'$, then for every $t\in \R$, $\widehat{\gamma'}(t+1)=\widehat{\gamma'}(t)+(1,0)$. Moreover, if $\epsilon>0$ is fixed and if $m$ is a sufficiently large integer, then for every $\widehat z$ very close to the upper boundary  the transverse path $I_{\widehat{\mathcal{F}}}^{\Z}(\widehat z)\vert_{[0,m]}$ contains a subpath equivalent to a translate of $\widehat{\gamma'}\vert_{[0,(2+\epsilon/2)m]}$ (this last property comes from the fact that the rotation number of the upper boundary is larger than $2$). Let $U_\Gamma$ be the annulus of the leaves crossed by $\Gamma$. Note that, as $\rot({\widehat g})$ has a negative real number, one can deduce that $U_\gamma$ does not intersect a neighborhood $V$ of the lower boundary component of $\Ac$.

Since $\A$ is a $SN$ mixed region of instability for $f$, it is also a $SN$ mixed region of instability for $g$ , we have that there exist points  $z_{S,N}$ and $z_{N,S}$, both lying in $V$, such that the $\omega$-limit set of $z_{S,N}$ for $g$ is contained in $\T{1} \times\{1\}$ and such that the $\alpha$-limit set for $g$ of $z_{N,S}$ is also contained in $\T1 \times\{1\}$.
This implies that there exists an integer $n_0$ sufficiently large  such that, for all integer $n\geq n_0$, $g^{n}(z_{S,N})$ and $g^{-n}(z_{N,S})$ are very close to the lower boundary. Let $\widehat{U_\Gamma}$ be a lift of $U_{\Gamma}$. Taking some lift $\widehat{z_{N,S}}$ of $z_{N,S}$, one get that, if $m$ is an integer sufficiently large, the transverse path $I_{\widehat{\mathcal{F}}}^\Z(\widehat{z_{N,S}})\vert_{[-m-n_0,0]}$ ends outside of $\widehat{U_\Gamma}$ but contains a subpath that is equivalent to a translate of $\widehat{\gamma'}\vert_{[0,(2+\epsilon/2)m]}$. Similarly, taking some lift $\widehat{z_{S,N}}$ of $z_{S,N}$, one get that, if $m$ is an integer sufficiently large, the transverse path $I_{\widehat{\mathcal{F}}}^\Z(\widehat{z_{S,N}})\vert_{[0,m+n_0]}$ starts outside of $\widehat{U_\Gamma}$ but contains a subpath that is equivalent to a translate of $\widehat{\gamma'}\vert_{[0,(2+\epsilon/2)m]}$. Each of those paths is admissible of order $n_0+m$. By Proposition \ref{proposition20LCT}, if $m$ large enough, one can construct a transverse path $\widehat \gamma$ that starts and ends outside of $\widehat{U_\Gamma}$, and that contains as a subpath a translate of $\widehat{\gamma'}\vert_{[0,(2+\epsilon/2)m]}$. The path $\widehat \gamma$ is admissible of order $2(n_0+m)$ and has a $\mathcal{F}$-transverse intersection with $\widehat \gamma +(j,0)$, where $j> (2+\epsilon/4)m$. Note that, as $n_0$ is fixed, if $m$ is sufficiently large, then $j>2(n_0+m)$. This ends the proof of Proposition \ref{pr:realizationSN}.

\subsection{Proof of Theorem D}

Before we start the proof of Theorem D, let us state a simple corollary of Proposition \ref{pr:realizationSN}.

\begin{coro}\label{cor:Birkhoffrealization}
Let $f:\Ac\to\Ac$ be a homeomorphism which is isotopic to the identity, and let $\widehat f$ be a lift of $f$ to the universal covering. Suppose that $f$ preserves a measure of full support, and that $\A=\T1\times(0,1)$ is a Birkhoff region of instability. For every $\rho$ in $\rot(\widehat{f})$ there exists a compact invariant set $Q_\rho$ such that for every point of $Q_\rho$ has a rotation number well-defined and it is equal to $\rho$. Moreover, if $\rho=p/q$ is a rational number, written in an irreducible way, then  $Q_\rho$ is the orbit of a period point of period $q$.
\end{coro}

\begin{proof}
Since $\T1\times(0,1)$ is a Birkhoff region of instability we can find, by Proposition \ref{pr:semi-mather}, that there exists an essential open annulus $A\subset \Ac$ which is a $SN$ mixed region of instability, and such that, if $A^{*}$ is the prime ends compactification of $A$, $f^{*}$ is the extension of $f$ to $A^{*}$, then there exists $\widehat{f^{*}}$ a lift of $f^{*}$ such that $\rot(\widehat{f^{*}})= \rot(\widehat{f})$.  By Proposition \ref{pr:realizationSN} we deduce that for any $\rho$ in  $\rot(\widehat{f})$ there exists a closed subset  $Q^{*}_{\rho}$ which is $f^{*}$-invariant and such that the $\widehat{ f^{*}}$ rotation number of any point in $Q^{*}_{\rho}$ is $\rho$. If $Q^{*}_{\rho}$ is contained in $A$, then it suffices to take $Q_{\rho}=Q^{*}_{\rho}$. Otherwise there exists a point in the boundary of $A^{*}$ with rotation number $\rho$, which implies that the rotation number of the restriction of $f^{*}$ to one of these boundaries is $\rho$. But this implies that there exists a connected component $K$ of the boundary of $A$ which is an essential continuum, such that the prime ends rotation number of $K$ is exactly $\rho$, and as the dynamics preserve a measure of full support, every point in $K$ has the same rotation number by Proposition \ref{pr:boundaryannulus}, in which case it suffices to take $K=Q_{\rho}$.
\end{proof}

\subsubsection{Proof of Theorem D}
In the following let $f$ be an area-preserving homeomorphism of $\Ac$ which is isotopic to the idenity. Let $\widehat{f}$ be a lift of $f$ to $\widehat{\Ac}$. If the rotation set of $\widehat{f}$ is a singleton, then the result is obvious as every point in the annulus will have the same rotation number. If not, then by a result from Franks (see \cite{franksannals1, franksannals2}), for every rational number $p/q$ in the interior of $\rot(\widehat{f})$ there exists a point $z_{p/q}$ in $\Ac$ whose rotation number is $p/q$. Therefore it suffices to show the result for $\rho$ in $\rot(\widehat{f})$ which is an irrational number. We can also assume that the rotation number of each boundary of $\Ac$ is not $\rho$, otherwise we are also done.

We will consider the extension $f'$ of $f$ to the open annulus $\A'=\T1\times\R$ by assuming that $f'(x,y)=f(x,1)+(0,y-1)$ if $y>1$ and  $f'(x,y)=f(x,0)+(0,y)$ if $y<0$, and by choosing a compatible lift $\widehat{f'}$ that is also an extension of $\widehat{f}$ to $\R^2$. We still denote $f$ to this extension. Let $(p_n/q_n)_{n\in\N}$ be a sequence of rational points in the interior of the rotation set of $\widehat{f}$ converging to $\rho$, and let $(z_n)_{n\in\N}$ be a sequence of periodic points of $\Ac$ such that for every integer $n$, $z_n$ has rotation number $p_n/q_n$. We assume that $(z_n)_{n\in\N}$ is converging, otherwise we take a subsequence, and let $\overline{z}$ be the limit of this sequence. We repeat a construction from \cite{KoropeckiTal2012StrictlyToral}. For each real number $\varepsilon>0$, consider the set
$$U'_{\varepsilon}=\bigcup_{i\in\Z}f^{i}(B_{\varepsilon}(\overline{z})),$$
where as usual $B_{\varepsilon}(z)$ is the $\varepsilon$ open ball centered at $z$. One can show, since $\overline{z}$ is non-wandering, that $U'_{\varepsilon}$ is an open set with finitely many connected components which are all periodic. Furthermore,  the connected component of $U'_{\varepsilon}$ that contains $\overline{z}$ cannot be contained in a topological disk, because every recurrent point in a periodic topological disk must have the same rational rotation number. As this component contains $z_n$ for any sufficiently large $n$, we would obtain a contradiction. This implies that the connected component $O_{\varepsilon}$ of $U'_{\varepsilon}$ that contains $\overline{z}$ (and consequently $B_{\varepsilon}(\overline{z})$) is essential. Since $f$ permutes the connected components of $U'_{\varepsilon}$, one has that either $O_{\varepsilon}$ is invariant or it is disjoint from its image by $f$. But the later cannot happen, as it would imply that every point in $O_{\varepsilon}$ is wandering since $O_{\varepsilon}$ is essential. Therefore $O_{\varepsilon}$ is invariant, and as it contains $B_{\varepsilon}(\overline{z})$ one deduces that $O_{\varepsilon}=U'_{\varepsilon}$. Finally note that, since $\overline{z}$ belongs to $\Ac$ and since $B_{\varepsilon}(\overline{z})\subset \T1\times [-\varepsilon, 1+\varepsilon]$, by the construction of the extension of $f$ one has that $U'_{\varepsilon}\subset \T1\times [-\varepsilon, 1+\varepsilon]$ and thus it separates the two ends of $\A'$.

Let $U_{\varepsilon}$ be the filling of $U'_{\varepsilon}$, that is, the union of $U'_{\varepsilon}$ with all connected components of its complement that are bounded in $\A'$.  Note that $U_{\varepsilon}$ is a topological open annulus, also contained in $\T1\times [-\varepsilon, 1+\varepsilon]$. Note also that $U_{\varepsilon_1}\subset U_{\varepsilon_2}$ if $\varepsilon_1<\varepsilon_2$. We will consider the set $K'_0=\bigcap_{\varepsilon\in(0,1)}\overline{U_{\varepsilon}}$. Now let $K_0$ be the filling of $K'_0$, and note that $K_0\subset \Ac$, that its complement contains only two connected components, $U_{+}$ which contains $\T1\times(1,\infty)$ and $U_{-}$, which contains $\T1\times(-\infty,0)$, and both these components are topological sub-annuli of $\A'$.

\begin{lemma}\label{lm:stayclose}
Given a neighborhood $W$ of $K_0$, there exists some $n_0=n_0(W)$ such that for all $n>n_0$ the orbit of $z_n$ remains in $W$ for all time.
\end{lemma}
\begin{proof}
We may assume that $W$ is open and contained in $\T1\times(-1,2)$, because $W\cap (\T1\times(-1,2))$ is also a neighborhood of $K_0$. Note that, if $\varepsilon<1$, then $\T1\times(-1,2)^C\subset \overline{U_{\varepsilon}}^C$. As $W^C\cap (\T1\times[-1,2])$ is compact and as $(\overline{U_{\frac{1}{n}}}^C)_{n\in\N}$ is an increasing sequence of open sets whose union contains $W^C$, one deduces that there must exists some $n_0>1$ such that $\overline{U_{\frac{1}{n_0}}}^C$ contains both $W^C\cap (\T1\times[-1,2])$ and $\T1\times(-1,2)^C$. Therefore $W \supset \overline{U_{\frac{1}{n_0}}}$. Since the later set is invariant and contains $z_n$ for sufficiently large $n$, the result follows.
\end{proof}

If $K_0$ has empty interior, then Theorem 2.8 of \cite{Koropecki} shows that the rotation number of any point in $K_0$ is the same, and this number must be $\rho$ by the previous lemma and Lemma \ref{lm:rotationnumber of the limit}, and so we can take $Q_{\rho}=K_0$. So we can assume that the interior of $K_0$ is not empty. If the interior of $K_0$ is inessential, then every point in $K_0$ has the same rotation number by Proposition \ref{pr:continuumwithinessentialinterior}, and this number must be $\rho$ by the previous lemma and Lemma \ref{lm:rotationnumber of the limit}, and so we can take $Q_{\rho}=K_0$. So we can assume that the interior of $K_0$ is essential. Therefore $\partial U_{-}$ is disjoint from $\partial U_{+}$, and their union is $\partial K_0$. We also remark that, as $K_0$ is contained in $\Ac$, the restriction of the dynamics to $K_0$ is nonwandering.

Consider now first the case where $\overline{z}$ belongs to $\partial K_0$.  We claim that in this case, the open sub-annuli $\A^*$ which is the interior of $K_0$ must be a Birkhoff region of instability. Let us first examine the case where $\overline{z}$ belongs to $\partial U_{-}$. If by contradiction the interior of $K_0$ is not a Birkhoff region of instability, as the restriction of $f$ to the interior of $K_0$ is nonwandering, one can find $V_{-}, V_{+}$ open invariant and disjoint neighborhoods of $\partial U_{-}$ and $\partial U_{+}$ respectively. But there must exists some $\delta>0$ such that $B_{\delta}(\overline{z})\subset V_{-}$ and therefore $U'_{\delta}$, and $U_{\delta}$, are disjoint from $V_{+}$. But this contradicts the fact that $\partial U_{+}\subset \overline{U_{\delta}}$. A similar argument shows that $\A^*$ is a Birkhoff region of instability if $\overline{z}\in\partial U_{+}$.

Let us consider two possibilities. First, assume that there exists an infinite subsequence $(z_{n_k})_{k\in\N}$ such that, for all $k$, $z_{n_k}$ does not belong to $K_0$. We can assume with no loss in generality that $z_{n_k}$ belong to $U_{+}$, and this implies that $\overline{z}$ belongs to $\partial U_{+}$. Note that every point in $\partial U_{+}$ has the same rotation number by Proposition \ref{pr:boundaryannulus}. Further note that if $W$ is a neighborhood of $\partial U_{+}$, then there exists a neighborhood $W'$ of $K_0$ such that $W\cap U_{+}=W'\cap U_{+}$ and therefore for sufficiently large $k$ one gets by Lemma \ref{lm:stayclose} that the whole orbit of $z_{n_k}$ is contained in $W$. This again implies that the rotation number of every point in $\partial U_{+}$ must be $\rho$ and we can take $Q_{\rho}=\partial U_{+}$.

Assume now that for all but finitely many $n\in\N$, that $z_n$ lie in $K_0$. Since $\partial K_0=\partial U_{-}\cup\partial U_{+}$, one deduces that the rotation number of points in $\partial K_0$ can have at most two values, and so we can assume that all $z_n$ lie in $\A^{*}$. Let $\widetilde f$ be the extension of $f$ to the prime ends compactification of $\A^{*}$, which is homeomorphic to $\Ac$ and $\widehat{\widetilde f}$ a lift of $\widetilde f$ which is compatible with $\widehat f$. Note that the rotation set of $\widehat{\widetilde f}$ contains $p_n/q_n$ for sufficiently large $n$ and, by being closed, must also contain $\rho$.  By Corollary \ref{cor:Birkhoffrealization} one finds a set $Q'_{\rho}$ in the prime ends compactification of $\A^{*}$. If $Q'_{\rho}$ lies in $\A^{*}$, we are done, just taking $Q_{\rho}=Q'_{\rho}$. If not, this implies that the prime end rotation number of $f'$ at one of the ends of $\A^{*}$ must be $\rho$. In this case, one deduces, again by Proposition \ref{pr:boundaryannulus}, that there exists a connected component $Q$ of the boundary of $\A^{*}$ in $\A$ such that every point in $Q$ has rotation number $\rho$ and we are done.

The final case we need to consider is when $\overline{z}$ belongs to $\A^*$. Again, if $\A^*$ is a Brikhoff region of instability, we can repeat the same argument as in the previous paragraph. Assume then that $\A^{*}$ is not a Birkhoff region of instability. In this case, one can find disjoint invariant neighborhoods $V_{-}$ and $V_{+}$ of the ends of $\A^*$. By eventually ``filling''  $V^{-}$ and $V^{+}$ with the connected components of its complement that are contained in $\A^{*}$, we can assume that both sets are essential open topological annuli.

We claim first that $\overline{z}$ belongs to the boundary of $V^{-}$. Indeed, if by contradiction one has that $\overline{z}$ is in the interior of $V^{-}$, then there exists some $\varepsilon>0$ such that $U_{\varepsilon}$ is contained in $V^{-}$, which contradicts $V^{+}\subset K_0$. Likewise, if $\overline{z}$ lies is in the interior of the complement of $V^{-}$, then again there exists $\varepsilon>0$ such that $U_{\varepsilon}$ is disjoint from $V^{-}$, which contradicts $V^{-}\subset K_0$. A similar argument shows that $\overline{z}\in \partial V^{+}$.  Let us then consider $K'=\partial V^{-}\cup\partial V^{+}$ and let $K$ be the union of $K'$ with the connected components of its complements that are contained in $\A^{*}$, and note that $\A^{*}\setminus K= V^{-}\cup V^{+}$ and that the interior of $K$ is inessential, otherwise  $\partial V^{-}\cap \partial V^{+}$ would be empty.  Finally, note that both $V^{-}$ and $V^{+}$ are Birkhoff regions of instability, the argument here being the same as in the case where $\overline{z}$ belonged to $\partial K_0$.

Since the interior of $K$ is inessential, we know that every point in it has the same rotation set. We can therefore assume, by possibly erasing a term of the sequence, that $z_n$ does not lie in $K$. Therefore either infinitely many of the points $z_n$ lie in $V^{-}$, or infinitely many of them lie in $V^{+}$.  But both $V^{-}$ and $V^{+}$ are Birkhoff regions of instability, and the same reasoning as in the case where $\overline{z}$ belonged to $\partial K_0$ and the $z_n$ lied in $\A^{*}$ can be applied to deduce the result.

\bibliographystyle{koro}
\bibliography{closedannulusjonathanbib}

  \textsc{Jonathan Conejeros. Departamento de Matem\'atica y Ciencia de la Computaci\'on, Universidad de Santiago de Chile, Avenida Libertador Bernardo O''Higgins 3363, Estaci\'on Central, Santiago, Chile}\\
  \textit{E-mail address:}  \texttt{jonathan.conejeros@usach.cl}\\

  \textsc{F\'abio Armando Tal. Instituto de Mat\'ematica e Estat\'{\i}stica, Universidade de S\~ao Paulo, Rua de Mat\~ao 1010, Cidade Universit\'aria, 05508-090 S\~ao Paulo, SP, Brazil}\\
  \textit{E-mail address:} \texttt{fabiotal@ime.usp.br}

\end{document}